\documentclass[10pt]{amsart}

\usepackage{amsmath,amssymb,amsthm}
\usepackage{mathrsfs}
\usepackage{geometry}
\usepackage{enumitem}
\usepackage{mathtools}
\usepackage{amscd}
\usepackage{esint}
\usepackage{cases}
\usepackage{empheq}
\usepackage{stmaryrd}
\usepackage{booktabs}
\usepackage{tabularx}
\usepackage{array}

\usepackage{tikz}
\usetikzlibrary{arrows.meta,positioning}
\usepackage{tikz-cd} % Added for commutative diagrams

\usepackage[dvipsnames]{xcolor}
\usepackage[
unicode=true,
colorlinks=true,
linktoc=all,
linkcolor=MidnightBlue,
citecolor=ForestGreen,
urlcolor=BrickRed,
filecolor=BrickRed,
bookmarksnumbered=true,
bookmarksopen=true
]{hyperref}
\hypersetup{
	pdftitle={CURVATURE AT INFINITY OF SCALAR-FLAT ALE FOUR-MANIFOLDS},
	pdfauthor={Jiangcheng You},
	pdfsubject={Differential Geometry, Geometric Analysis},
	pdfkeywords={ALE manifolds, scalar-flat metrics, harmonic curvature, Bach-flat metrics, Brown--York mass, ADM mass}
}

\newcommand{\R}{\mathbb{R}}
\newcommand{\N}{\mathbb{N}}
\newcommand{\Sph}{\mathbb{S}}
\newcommand{\gE}{g_{\mathrm{E}}}

\newcommand{\del}{\partial}
\newcommand{\tr}{\operatorname{tr}}
\newcommand{\divE}{\operatorname{div}_{\gE}}
\newcommand{\Bach}{\operatorname{B}}
\newcommand{\Ric}{\operatorname{Ric}}
\newcommand{\Rm}{\operatorname{Rm}}

\newcommand{\Weyl}{\operatorname{W}}
\newcommand{\Hom}{\operatorname{Hom}}

\newcommand{\Id}{\operatorname{Id}}

\newcommand{\cyc}{\mathrm{cyc}}
\newcommand{\Scal}{\operatorname{R}}
\newcommand{\Lie}{\mathcal{L}}

\newcommand{\DeltaE}{\Delta_{\gE}}
\newcommand{\trE}{\operatorname{tr}_{\gE}}

\newcommand{\lin}{\mathrm{lin}}

\newcommand{\BE}{\mathcal{B}_{\gE}}

\newcommand{\ADM}{\mathrm{ADM}}

\newcommand{\Sch}{\mathrm{Sch}}

\newcommand{\SO}{\mathrm{SO}}

\newcommand{\divergence}{\mathrm{div}}
\newcommand{\SU}{\mathrm{SU}}

\newcommand{\Burns}{\mathrm{Burns}}

\theoremstyle{plain}
\newtheorem{theorem}{Theorem}[section]
\newtheorem{proposition}[theorem]{Proposition}
\newtheorem{lemma}[theorem]{Lemma}
\newtheorem{corollary}[theorem]{Corollary}

\newtheorem{definition}[theorem]{Definition}
\newtheorem{assumption}[theorem]{Assumption}
\newtheorem{remark}[theorem]{Remark}

\title{CURVATURE AT INFINITY OF SCALAR-FLAT ALE FOUR-MANIFOLDS}

\author{Jiangcheng You}
\address{School of Mathematical Sciences, University of Science and Technology of China, Hefei 230026, China}
\email{yjcmp@mail.ustc.edu.cn}

\makeatletter

\def\printead#1{
	\par\smallskip\noindent\textsc{\@author}\par 
	\textit{\@address}\par 
	\textit{E-mail address:} \texttt{\@email}\par 
}
\makeatother

\begin{document}

	\begin{abstract}
		We study refined asymptotics of scalar-flat ALE four-manifolds in the
		Tian--Viaclovsky setting, namely for self-dual or anti-self-dual metrics and
		for metrics with harmonic curvature.  Starting from the ALE coordinates
		obtained by Tian--Viaclovsky, we construct preferred coordinates at infinity
		and identify the homogeneous $|x|^{-2}$ term in the metric expansion.  This
		term splits canonically into a scalar part determined by the ALE ADM mass and
		an algebraic Weyl tensor at infinity. As an application, we consider scalar-flat K\"ahler ALE metrics on minimal
		resolutions $\pi:X\to\mathbb C^2/\Gamma$ of quotient surface singularities.
		In this case, the leading Weyl tensor at infinity vanishes exactly
		when the minimal resolution is crepant.
	\end{abstract}
	
	\maketitle

	\section{Introduction}\label{sec:intro}
	
	The purpose of this paper is to study the refined asymptotic geometry of
	scalar-flat ALE ends arising from the two curvature classes in the
	Tian--Viaclovsky framework \cite{TV}: self-dual or anti-self-dual metrics, and metrics
	with harmonic curvature. At a first level, an ALE end is a
	region which, after passing to a finite quotient, becomes asymptotic to the
	Euclidean end. The main question here is what can be said beyond this basic
	Euclidean approximation: whether one can choose better coordinates at infinity,
	identify the first nontrivial term in the metric expansion, and determine which
	part of this term is intrinsic to the geometry of the end.
	
	We begin by recalling the definition of an ALE end, following
	\cite[(1.3)--(1.4)]{TV}. Let $(M,g)$ be a smooth noncompact Riemannian
	$4$--manifold. An end $\Omega\subset M$ is said to be \emph{ALE of order
		$\tau>0$} if there exist $R>0$, a finite subgroup
	$\Gamma\subset \SO(4)$ acting freely on $\R^4\setminus B_R$, where
	$B_R=B(0,R)$, and a $C^\infty$ diffeomorphism
	\begin{align*}
		\Phi:\Omega \longrightarrow (\R^4\setminus B_R)/\Gamma
	\end{align*}
	such that, if $\pi_\Gamma:\Omega_R:=\R^4\setminus B_R \to (\R^4\setminus B_R)/\Gamma$ denotes the quotient map, then the pulled-back metric $\tilde g:=(\Phi^{-1}\circ \pi_\Gamma)^*g$
	satisfies
	\begin{align*}
		\tilde g_{ij}
		=
		\delta_{ij}+O(|x|^{-\tau}),
		\qquad
		\partial^\alpha \tilde g_{ij}
		=
		O(|x|^{-\tau-|\alpha|})
	\end{align*}
	for every multi-index $\alpha$.
	
	\medskip
	
	Our starting point is the exterior chart obtained by Tian--Viaclovsky in \cite{TV}. They carried out the following assumptions:
	
	\begin{assumption}\label{ass:TVbaseline}
		Let $(M,g)$ be a complete and noncompact $4$--manifold. Assume that $(M,g)$ satisfies the
		following hypotheses:
		\begin{enumerate}
			\item the scalar curvature $\Scal_g$ vanishes on $M$;
			\item $\int_M |\Rm_g|^2\,dV_g<\infty$;
			\item the Sobolev constant $C_S(M,g)$ is finite;
			\item the first Betti number $b_1(M)<\infty$;
			\item 
			\begin{enumerate}
				\item $g$ is self-dual or anti-self-dual, or
				\item $g$ has harmonic curvature,
			\end{enumerate}
		\end{enumerate}
	\end{assumption}
	
	and proved
	
	\begin{theorem}[\cite{TV}]\label{thm:baseline-ALE}
		Assume Assumption~\ref{ass:TVbaseline}. Then $M$ has finitely many ends, and each end is ALE
		of order $\tau$ for every $\tau<2$.
	\end{theorem}
	
	\medskip
	
	We now fix one end $\Omega$ of $M$, choose an exterior chart $\Phi:\Omega\to (\R^4\setminus B_R)/\Gamma$ and work on the corresponding $\Gamma$-cover $\Phi^{-1}\circ\pi_\Gamma:\Omega_R\to\Omega$. All asymptotic notation below is understood in these fixed coordinates, with respect to the
	Euclidean background $(\gE)_{ij}=\delta_{ij}$ on $\Omega_R$, and $r:=|x|$ always denotes the Euclidean radius. In particular,
	$\mathcal O_\infty(|x|^{-k})$ means that for every multi-index $\alpha$, there holds $\partial^\alpha f = O(|x|^{-k-|\alpha|})$.
	
	\medskip
	
	We shall use both alternatives in Assumption~\ref{ass:TVbaseline}(5). Their
	relation with Bach-flatness is different. In dimension $4$, a self-dual or anti-self-dual metric is Bach-flat. The Bach tensor is given by
	\begin{align*}
		\Bach(g)_{ij}
		=
		\nabla^k\nabla^\ell \Weyl(g)_{ikj\ell}
		+
		\frac12\,\Ric(g)^{k\ell}\,\Weyl(g)_{ikj\ell}.
	\end{align*}
	The second alternative, namely harmonic curvature, is not the same as
	Bach-flatness in general. However, together with $\Scal_g=0$, it implies the vanishing of the Cotton tensor, and this condition can also be treated by the same fourth-order analysis at infinity. This point will be explained in Section~\ref{sec:gauge}.
	
	\subsection{Main results and related works}\label{sec:main123}
	
	In the first part of our paper, we construct a preferred coordinate system
	near infinity for scalar-flat ALE manifolds satisfying
	Assumption~\ref{ass:TVbaseline}. This provides a clearer understanding of
	the geometry at infinity. This study is related to several developments in ALE geometry, particularly concerning refined asymptotics at infinity.
	
	Under fast curvature decay and maximal volume growth, Bando--Kasue--Nakajima
	constructed global coordinates at infinity and proved ALE structure \cite{BKN}.
	In the Ricci-flat setting, Cheeger--Tian studied asymptotic cones under Euclidean
	volume growth and critical integral curvature bounds \cite{CT94}. In dimension $4$,
	Kronheimer constructed ALE hyper-K\"ahler spaces, relating ALE geometry to quotient
	surface singularities and their resolutions \cite{Kro89}. Minerbe further studied
	gravitational instantons from the viewpoint of volume growth, including the ALF case
	arising from cubic volume growth \cite{Minerbe}.
	
	More refined asymptotic invariants have been studied in the Einstein and Ricci-flat ALE
	setting. Biquard--Hein introduced the renormalized volume of four-dimensional
	Ricci-flat ALE spaces and proved its sharp nonpositivity and rigidity \cite{BH}.
	Wang--Yin later constructed preferred coordinates at infinity, defined a Weyl tensor at
	infinity, and related these data to the renormalized volume for Ricci-flat ALE spaces
	and orbifolds \cite{WY}. These works show that the leading terms at infinity can carry
	intrinsic geometric information.
	
	ALE K\"ahler geometry provides another important source of examples and comparison
	results. Besides Kronheimer's hyper-K\"ahler spaces, Calderbank--Singer constructed
	complete ALE scalar-flat K\"ahler metrics on toric resolutions of isolated cyclic quotient
	singularities \cite{CS04}. Further existence, deformation, compactness, and
	classification results for ALE K\"ahler manifolds and scalar-flat K\"ahler ALE surfaces
	were developed by Lock--Viaclovsky, Han--Viaclovsky, and
	Hein--R\u{a}sdeaconu--\c{S}uvaina; see for example \cite{LV19,HV,HV20,HRS}.
	Hein--LeBrun proved a mass formula in ALE K\"ahler geometry and showed that, in the
	scalar-flat K\"ahler ALE case, the mass has a topological interpretation \cite{HL}.
	Related analytic and metric characterizations of ALE K\"ahler manifolds were obtained
	by Ni--Shi--Tam \cite{NST}.
	
	The analytic framework most directly used in this paper comes from
	Tian--Viaclovsky's work on four-dimensional critical metrics. Their results give
	curvature decay, ALE structure, and compactness for Bach-flat metrics and metrics with
	harmonic curvature under finite $L^2$ curvature, Sobolev control, and mild topological
	assumptions; see \cite{TV,TV05b,TV08,VPCMI}. For related background, see also
	\cite{Anderson06}. We use this ALE framework as the starting point, but our aim is to obtain a finer description of the leading asymptotic terms in the scalar-flat ALE setting covered by Assumption~\ref{ass:TVbaseline}.
	
	\medskip
	
	Before stating the main results, we fix some notations. For $\lambda\in\R$, let
	$\Xi(\lambda)$ denote the scalar pair-symmetric tensor defined by
	\begin{align*}
		\Xi(\lambda)_{ijk\ell}
		=
		-\frac{\lambda}{9}\delta_{ij}\delta_{k\ell}
		+
		\frac{2\lambda}{9}
		\bigl(\delta_{ik}\delta_{j\ell}+\delta_{i\ell}\delta_{jk}\bigr).
	\end{align*}
	
	The first main theorem also involves a normalized Schwarzschild reference metric, denoted by $g^{\Sch,\sharp}_\lambda$. Since throughout our study, we work on a fixed ALE end $\Omega$ guaranteed by Theorem~\ref{thm:baseline-ALE}; hence, the ADM mass $m_\ADM$ of this end and the asymptotic group $\Gamma$, in view of $\Phi:\Omega\to (\R^4\setminus B_R)/\Gamma$, are fixed. Set $m_{\Sch}:=|\Gamma|m_{\ADM}$, and the isotropic Schwarzschild metric $g^{\Sch}:=U^2 \gE$ on the cover $\R^4\setminus B_R$, where $U(x):=1+\frac{m_{\Sch}}{2r^2}$. Then the Schwarzschild reference metric $g^{\Sch,\sharp}$ is defined as the pull-back of the metric $g^{\Sch}$ via the radial change of variables $\Phi^{\Sch}(x):=x-m_\Sch\frac{x}{r^2}$, i.e.,
	\begin{align*}
		g^{\Sch,\sharp}:=(\Phi^{\Sch})^*g^{\Sch}_\lambda.
	\end{align*}
	
	We now state the first two main results.
	
	\begin{theorem}\label{thm:main1}
		Let $(M^4,g)$ be a complete noncompact Riemannian manifold satisfying Assumption~\ref{ass:TVbaseline}, and fix an ALE end
		$\Omega$ modeled on $(\R^4\setminus B_{R_0})/\Gamma$, where $\Gamma$ is a finite subgroup of $\SO(4)$ acting freely on $\R^4\setminus B_{R_0}$. Then there exist a number $\sigma\in(0,1)$ and $R\geq R_0$ such that, after lifting to the
		$\Gamma$-cover of the end, there is a diffeomorphism
		\begin{align*}
			\Psi:\widetilde{\Omega}\longrightarrow \R^4\setminus B_R
		\end{align*}
		with the following properties. If $x=\Psi(\cdot)$ and $g_{ij}(x)$ denotes the expression of the metric $g$ in these coordinates, then
		\begin{enumerate}
			\item The metric admits the expansion
			\begin{align}\label{eq:intro-main-expansion}
				g_{ij}(x)
				=
				\delta_{ij}
				+
				\left( \Xi\bigl(9|\Gamma|\,m_{\ADM}(g)\bigr)_{ijk\ell}
				+
				(\Weyl_\infty)_{ik\ell j}\right)\frac{x_kx_\ell}{|x|^4}
				+
				\mathcal O_\infty(|x|^{-2-\sigma}),
			\end{align}
			where $\Weyl_\infty$ is the algebraic Weyl tensor and $m_\ADM$ is the ADM mass of the end.
			\item The metric $g$ and the Schwarzschild reference metric $g^{\Sch,\sharp}$ satisfy the following relative gauge condition:
			\begin{align*}
				g^{ij}\Bigl(\Gamma(g)^k_{ij}-\Gamma(g^{\Sch,\sharp})^k_{ij}\Bigr)=0
				\quad \text{for any}\ k.
			\end{align*}
		\end{enumerate}
	\end{theorem}
	
	\begin{remark}
		Tian--Viaclovsky's theorem in \cite{TV} gives an exterior coordinate system in
		which
		\begin{align*}
			g_{ij}
			=
			\delta_{ij}
			+
			\mathcal O_\infty(r^{-\tau})
			\qquad
			\text{for all } \tau<2 .
		\end{align*}
		Theorem~\ref{thm:main1} starts from this ALE structure and refines the
		borderline order $r^{-2}$. After a further choice of preferred coordinates, we identify the first non-trivial term in the expansion of the metric $g$ as:
		\begin{align*}
			\left( \Xi\bigl(9|\Gamma|\,m_{\ADM}(g)\bigr)_{ijk\ell}
			+
			(\Weyl_\infty)_{ik\ell j}\right)\frac{x_kx_\ell}{|x|^4}
		\end{align*}
		In this sense, our result is a refinement beyond the Tian--Viaclovsky ALE decay estimate: the mass part and the Weyl-type
		part of the asymptotic geometry are separated explicitly.
	\end{remark}
	
	\begin{remark}\label{rem:intro-WY-main1}
		Theorem~\ref{thm:main1} is inspired by the preferred-coordinate construction
		of Wang--Yin for Ricci-flat ALE ends \cite{WY}. In their setting, the leading
		term in a Bianchi coordinate is purely of Weyl type. In the present scalar-flat setting under Assumption~\ref{ass:TVbaseline},
		the same Weyl-type mechanism remains, but an additional scalar component appears. Thus Theorem~\ref{thm:main1} can be viewed as an analogue of the Wang--Yin expansion in the Tian--Viaclovsky ALE setting, with the extra mass term separated explicitly.
	\end{remark}
	
	\begin{remark}\label{rem:intro-Shiromizu-Tomizawa}
		The decomposition in \eqref{eq:intro-main-expansion} echoes an observation of Shiromizu--Tomizawa on spatial infinity in higher-dimensional
		spacetimes \cite{ShiromizuTomizawa}. In spacetime dimensions higher than four,
		the geometry at spatial infinity need not be maximally symmetric. In the
		static vacuum setting, the Geroch-type multipole moments do not, in general,
		determine the full local asymptotic structure; additional data are needed, at
		least a Weyl-type tensor intrinsic to the geometry at spatial infinity.
		
		Theorem~\ref{thm:main1} gives an ALE Riemannian analogue of this phenomenon. After the preferred normalization, the leading asymptotic data split into an ADM mass part and a Weyl-type part $\Weyl_\infty$. The latter should be viewed as additional boundary data at infinity, not captured by the mass alone.
	\end{remark}
	
	\medskip
	
	We also need to know whether the preferred coordinates are essentially unique. The
	next result gives this uniqueness statement. It says that if two coordinate systems both
	have the form described in Theorem~\ref{thm:main1}, then their transition map is
	asymptotically a Euclidean motion. Consequently, the leading coefficient in
	\eqref{eq:intro-main-expansion} is well-defined up to the natural orthogonal action.
	
	\begin{theorem}\label{thm:main2}
		Assume $(M,g)$ and $\Omega$ satisfy the hypotheses of Theorem~\ref{thm:main1}. Let $x$ and $y$ are two coordinate systems on $\widetilde \Omega$ such that each of them satisfies Theorem~\ref{thm:main1} (1)--(2). Specifically, the metric $g$ admits expansions in the two coordinate systems:
		\begin{align*}
			g_{ij}^{(x)}(x)
			&=
			\delta_{ij}
			+
			\left( \Xi\bigl(9|\Gamma|\,m_{\ADM}(g)\bigr)_{ijk\ell}
			+
			(\Weyl_\infty^{(x)})_{ik\ell j}\right)\frac{x_kx_\ell}{|x|^4}
			+
			\mathcal O_\infty(|x|^{-2-\sigma}),
			\\
			g_{ab}^{(y)}(y)
			&=
			\delta_{ab}
			+
			\left( \Xi\bigl(9|\Gamma|\,m_{\ADM}(g)\bigr)_{abpq}
			+
			(\Weyl_\infty^{(y)})_{apqb}\right)\frac{y_py_q}{|y|^4}
			+
			\mathcal O_\infty(|y|^{-2-\sigma})
		\end{align*}
		for some $\sigma\in(0,1)$. Then there exists an $A\in \mathrm O(4)$ such that
		\begin{align*}
			(\Weyl_\infty^{(x)})_{ijk\ell}
			=
			A_{ai}A_{bj}A_{ck}A_{d\ell}\,
			(\Weyl_\infty^{(y)})_{abcd}.
		\end{align*}
	\end{theorem}
	
	\begin{remark}\label{rem:intro-WY-main2}
		Theorem~\ref{thm:main2} shows that the leading data in
		\eqref{eq:intro-main-expansion} are independent of the remaining choice of
		preferred coordinates, up to the natural orthogonal action. This is analogous
		to the uniqueness result of Wang--Yin for the Weyl tensor at infinity in the
		Ricci-flat case \cite{WY}. In our setting, the mass coefficient is fixed, and
		only $\Weyl_\infty$ changes by the $\mathrm O(4)$--action. Thus
		$\Weyl_\infty$ is well-defined as asymptotic data of the end, modulo this
		orthogonal ambiguity.
	\end{remark}
	
	\medskip
	
	A useful test example is the Burns metric $g^{\Burns}$ on $\mathrm{Bl}_0\mathbb C^2$.  This
	is a complete scalar-flat K\"ahler ALE metric with trivial group at infinity;
	see \cite{LeBrun88,LeBrun91,HL}. Since a scalar-flat K\"ahler surface is anti-self-dual with respect to the
	complex orientation, the Burns metric belongs to the first branch of
	Assumption~\ref{ass:TVbaseline}. In this example the
	preferred-coordinate expansion \eqref{eq:intro-main-expansion} can be computed explicitly: the coefficient of
	the $r^{-2}$ term is
	\begin{align*}
		\Xi(3)_{ijk\ell}+(\Weyl_\infty^{\mathrm{Burns}})_{ik\ell j},
		\quad
		\Weyl_\infty^{\mathrm{Burns}}\neq 0.
	\end{align*}
	Since $\Gamma=\{1\}$, this gives $m_{\ADM}(g^{\Burns})=1/3$. Thus the Burns metric $g^{\Burns}$ shows that
	the leading asymptotic term in \eqref{eq:intro-main-expansion} may contain both a nonzero mass part and a nonzero
	Weyl-type part.  It also illustrates Theorem~\ref{thm:main2}: the mass part is
	unchanged under a change of preferred coordinates, while the Weyl tensor at
	infinity transforms by the natural $\mathrm O(4)$--action.  The detailed
	calculation is given in Subsection~\ref{subsec:burns-example}.
	
	\begin{figure}[htbp]
		\centering
		\begin{tikzpicture}[scale=0.95,>=Latex]
			% The blown-up manifold
			\draw[fill=gray!12,draw=black!70]
			(-2.6,1.1)
			.. controls (-2.0,1.65) and (-1.0,1.38) .. (-0.32,0.5)
			.. controls (-0.12,0.24) and (0.12,0.24) .. (0.32,0.5)
			.. controls (1.0,1.38) and (2.0,1.65) .. (2.6,1.1)
			-- (2.6,-1.1)
			.. controls (2.0,-1.65) and (1.0,-1.38) .. (0.32,-0.5)
			.. controls (0.12,-0.24) and (-0.12,-0.24) .. (-0.32,-0.5)
			.. controls (-1.0,-1.38) and (-2.0,-1.65) .. (-2.6,-1.1)
			-- cycle;
			
			% Exceptional divisor
			\draw[fill=red!24,draw=red!70!black,thick] (0,0) circle (0.43);
			
			% Large coordinate sphere
			\draw[blue!70!black,dashed,thick] (1.95,0) ellipse (0.78 and 1.2);
			
			% Arrow indicating the ALE end
			\draw[-{Latex[length=2.8mm]},thick] (3.05,1.45) -- (2.35,0.82);
			
			% Labels
			\node[align=center] at (-1.75,1.82)
			{$\mathrm{Bl}_0\mathbb C^2$};
			
			\node[red!70!black,align=center] at (0,-1.67)
			{$E\simeq\mathbb{CP}^1$};
			
			\node[blue!70!black,align=center] at (2.85,-1.67)
			{large sphere\\[-1mm] $|y|=R$};
			
			\node[align=center] at (3.23,1.82)
			{ALE end\\[-1mm] $\simeq\mathbb R^4\setminus B_R$};
		\end{tikzpicture}
		\caption{The Burns metric on $\mathrm{Bl}_0\mathbb C^2$.  The exceptional
			divisor replaces the blown-up origin, while the metric is ALE at infinity.
			This example shows that the leading term in
			Theorem~\ref{thm:main1} can have both a mass part and a Weyl-type part.}
		\label{fig:intro-Burns-example}
	\end{figure}
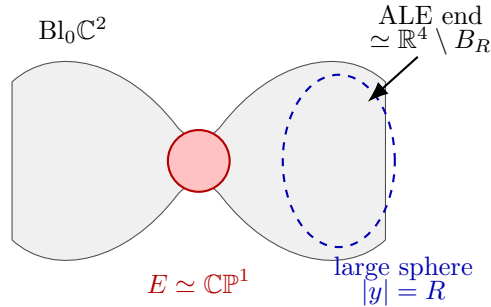
	
	\medskip
	
	Theorems~\ref{thm:main1} and~\ref{thm:main2} associate to a scalar-flat ALE
	four-manifold a preferred homogeneous $|x|^{-2}$ coefficient at infinity, as \eqref{eq:intro-main-expansion}. As an application, we explain what this Weyl-type part records in the scalar-flat K\"ahler ALE setting, when the complex surface is the minimal resolution of a quotient surface singularity.
	
	Fix a finite group $\Gamma\subset\SO(4)$ from the ALE setting above. Moreover, we assume that $\Gamma$ preserves a Euclidean complex structure, denoted by $J_\infty$, on
	$\R^4$.  After identifying this complex vector space with $\mathbb C^2$, we may
	regard $\Gamma$ as a finite subgroup of $\mathrm U(2)$.  Since $\Gamma$ acts
	freely on $\Sph^3$, it contains no complex reflections, and the quotient
	$\mathbb C^2/\Gamma$ has an isolated quotient singularity at the image of the
	origin.
	
	Let
	\begin{align*}
		\pi:X\longrightarrow \mathbb C^2/\Gamma
	\end{align*}
	be the minimal resolution of this singularity, and denote by $J$ the complex
	structure on $X$.  The quotient complex structure at infinity is identified with
	$J$ on the end through the biholomorphism
	$\pi:X\setminus \pi^{-1}(0)\to \mathbb C^2/\Gamma\setminus\{0\}$.
	
	We recall the notion of minimal resolution in the form used below; see
	\cite{BHPV,LV19}.
	
	\begin{definition}[Minimal resolution of $\mathbb C^2/\Gamma$]
		Let $\Gamma\subset\mathrm U(2)$ be a finite subgroup containing no complex
		reflections.  A pair $(X,\pi)$ is called the minimal resolution of
		$\mathbb C^2/\Gamma$ if $X$ is a smooth complex surface and
		\begin{align*}
			\pi:X\longrightarrow \mathbb C^2/\Gamma
		\end{align*}
		is a proper holomorphic map such that, writing $E=\pi^{-1}(0)$, the restriction
		\begin{align*}
			\pi:X\setminus E\longrightarrow \mathbb C^2/\Gamma\setminus\{0\}
		\end{align*}
		is biholomorphic, and $E$ contains no $(-1)$-curves.  The divisor $E$ is called
		the exceptional divisor.
	\end{definition}
	
	The minimal resolution is unique up to biholomorphism.  For quotient surface
	singularities, the exceptional divisor is a connected configuration of rational
	curves with negative definite intersection matrix.  In the cyclic case it is a
	Hirzebruch--Jung string.  In the non-cyclic case it is a tree with one central
	rational curve and three Hirzebruch--Jung strings attached to it
	\cite{BHPV,LV19}.
	
	The resolution map gives $X$ a natural end.  Since $\pi$ is biholomorphic away
	from the compact exceptional divisor $E$, for large $R$ the set
	\begin{align*}
		X_\infty(R)
		=
		\pi^{-1}\bigl((\mathbb C^2\setminus B_R)/\Gamma\bigr)
	\end{align*}
	is an end of $X$, and $\pi$ identifies it with
	$(\mathbb C^2\setminus B_R)/\Gamma$.  After lifting this end to the
	$\Gamma$-cover, we get standard holomorphic coordinates $z=(z^1,z^2)$ on
	$\mathbb C^2\setminus B_R$.
	
	There are several existence results for scalar-flat K\"ahler ALE metrics in this
	setting.  Kronheimer constructed the hyperk\"ahler ALE metrics in the ADE case
	$\Gamma\subset\SU(2)$ \cite{Kro89}.  Calderbank--Singer constructed scalar-flat
	K\"ahler ALE metrics on resolutions of cyclic quotient singularities
	\cite{CS04}.  Lock--Viaclovsky proved the corresponding existence result on
	minimal resolutions for non-cyclic finite subgroups $\Gamma\subset\mathrm U(2)$
	without complex reflections \cite[Theorem~1.3]{LV19}. We fix a scalar-flat K\"ahler ALE metric
	$g$ on $(X,J)$ whose ALE end is the quotient end determined by $\pi$, and we
	study the leading Weyl tensor at infinity obtained from that end.
	
	Our third main result can be stated as
	
	\begin{theorem}\label{thm:main3}
		Let $\Gamma\subset\mathrm U(2)$ be a finite subgroup acting freely on
		$\Sph^3$, set $Y=\mathbb C^2/\Gamma$, and let $\pi:X\to Y$ be the minimal
		resolution.  Let $g$ be a scalar-flat K\"ahler ALE metric on $(X,J)$ whose end
		is the quotient end determined by $\pi$.  For any preferred coordinate system
		$x$ obtained by applying Theorem~\ref{thm:main1} starting from the holomorphic
		ALE coordinates induced by $\pi$, let $\Weyl_\infty^{(x)}(g)$ be the Weyl
		tensor at infinity in these coordinates.  Then the following conditions are
		equivalent:
		\begin{enumerate}[label=\textup{(\arabic*)}]
			\item $\Weyl_\infty^{(x)}(g)=0$ for one, equivalently for every, such
			preferred coordinate system $x$;
			\item $K_X=\pi^*K_Y$.
		\end{enumerate}
		In other words, the leading Weyl tensor at infinity vanishes exactly when the
		minimal resolution is crepant.
	\end{theorem}
	
	\begin{remark}
		In the present quotient setting, the crepant case is the ADE case.  Since
		$\Gamma$ acts freely on $\Sph^3$, it contains no complex reflections.  Hence,
		by Watanabe's theorem on invariant subrings, the quotient
		$\mathbb C^2/\Gamma$ is Gorenstein precisely when
		$\Gamma\subset\mathrm{SL}(2,\mathbb C)$ \cite{Watanabe74}.  Because
		$\Gamma\subset\mathrm U(2)$, this is equivalent to
		$\Gamma\subset\mathrm{SU}(2)$ with respect to the chosen complex structure at
		infinity. For nontrivial $\Gamma$, these quotients are exactly the rational
		double points, or Du Val singularities, with ADE exceptional configurations
		\cite{BHPV,Kro89}; the trivial group gives the smooth case $Y=\mathbb C^2$.
	\end{remark}

	\subsection{Idea of proof}
	
	The proof is organized around two ideas. The first is to separate the geometric
	part of the leading coefficient in the metric expansion at infinity. Starting
	from the Tian--Viaclovsky ALE chart, we put the metric into Euclidean Bianchi
	gauge. In this gauge, the equations coming from the two alternatives in
	Assumption~\ref{ass:TVbaseline} become a fourth-order elliptic system whose principal part is the Euclidean bi-Laplacian. This gives a
	first expansion at the borderline order $r^{-2}$, but the leading term coefficient, denote by $A$, obtained at
	this stage still contains coordinate freedom.
	
	The next step is algebraic. The point is not to guess the final normal form directly. Instead, we start from the algebraic properties that a Weyl-type tensor must satisfy, as detailed in Proposition~\ref{prop:Wtilde-eq}. Under the map $s$ defined as \eqref{eq:def-s-map}, a Weyl
	tensor gives a pair-symmetric tensor satisfying two more conditions: vanishing partial traces, and the first Bianchi identity. We therefore use these conditions as a guide to trace the Weyl-type part inside $A$.
	
	The pair symmetry is already present in the rough expansion. The main issue appears when we try to achieve the partial trace-free property by using coordinate transformation
	\begin{align}\label{eq:coordinate-transformation-introduction}
		x=\varphi(y)=y+u(y),
		\quad 
		u_i(y)=B_{ij}\tfrac{y_j}{|y|^2},
	\end{align}
	where the constant matrix $B$ is to be determined, to eliminate the remaining coordinate freedom. The target partial trace-free equations cannot, in general, be solved completely. A scalar
	obstruction remains (c.f. Lemma~\ref{lem:LA-Psi4-trace-normalization} and Proposition~\ref{prop:residual-charge-vanishes}). This obstruction is not an unwanted error; rather, it is precisely the component that cannot be eliminated by the above coordinate transformation, thereby holding the promise of being an intrinsic geometric quantity on the end. In fact, it arises from our tracking of the Weyl-type tensor and can be identified with the ALE ADM mass.
	
	Thus, by choosing a suitable constant matrix $B$ and performing the coordinate transformation again, the residual partial trace-free coefficient satisfies the first Bianchi identity and therefore comes from an algebraic Weyl tensor $\Weyl_\infty$ at infinity by Proposition~\ref{prop:Wtilde-eq}. At this stage, the leading coefficient $A$ is transformed into $A^\sharp$, which decomposes into a mass part and a Weyl-type part, both of which are preferred data:
	\begin{align*}
		A^\sharp
		=
		\Xi\bigl(9|\Gamma|\,m_{\ADM}(g)\bigr)
		+
		s(\Weyl_\infty).
	\end{align*}
	This is the main mechanism behind Theorem~\ref{thm:main1}. The uniqueness of the
	preferred data is then obtained by comparing two such elimination normalizations; the
	transition map is asymptotically orthogonal, so only the natural $\mathrm O(4)$--action remains.
	
	It is evident that the key operation in carrying out the first idea is the
	coordinate transformation \eqref{eq:coordinate-transformation-introduction}. The
	coordinate-change technique used here is adapted from the work of Wang--Yin on
	Ricci-flat ALE ends \cite{WY}. Their argument combines suitable decaying
	coordinate transformations with algebraic decompositions of the leading
	coefficient to obtain a Weyl-type normal form. We use this strategy in the present scalar-flat ALE setting, where the normalization problem has an additional
	feature: besides the Weyl-type component, a scalar component remains, and this
	remaining scalar component is identified with the ADM mass.
	
	The second idea enters when the ALE space is scalar-flat K\"ahler and the end is
	the holomorphic quotient end coming from the minimal resolution
	$\pi:X\to \mathbb C^2/\Gamma$. In this case, we carry out the same computational procedure as in the proof of Theorem~\ref{thm:main1}, but now keeping track of the K\"ahler structure. The $|x|^{-2}$ coefficient in the modified metric expansion can then
	be computed more explicitly. Its scalar part is exactly the ADM mass term
	already isolated in Theorem~\ref{thm:main1}; after this mass contribution is
	removed, the remaining preferred component must be the Weyl-type part.  Of course, it is not difficult to see that this residual Weyl tensor is determined by the limiting complex structure $J_\infty$ on $\mathbb{R}^4$.  Thus, in preferred coordinates
	coming from the holomorphic ALE end, one obtains
	\begin{align*}
		\Weyl_\infty^{(x)}(g)
		=
		C_W\,\mathcal W_{J_\infty^{(x)}} ,
	\end{align*}
	where the coefficient $C_W$ is, at this stage, an explicit universal multiple of
	the ADM mass.
	
	The next task is to interpret this ADM mass geometrically.  Here we use the mass formula of
	Hein--LeBrun \cite{HL}: for scalar-flat K\"ahler ALE surfaces, the ADM mass is
	expressed in terms of the pairing between the first Chern class of the complex
	surface and the K\"ahler class.  Since our end comes from the minimal resolution
	of the quotient singularity, this pairing can be rewritten on the
	exceptional set by using the discrepancy of the resolution. This gives
	\begin{align*}
		C_W
		=
		-\frac{|\Gamma|}{\pi}\sum_i\mu_iA_i ,
	\end{align*}
	where the $\mu_i$ are the discrepancy coefficients and $A_i=\int_{E_i}\omega$ are the K\"ahler areas of the exceptional curves. Meanwhile, the non-crepant property of the minimal resolution can also be characterized by a linear combination of the exceptional divisors, with the corresponding coefficients being precisely $\mu_i$.
	
	The proof of Theorem~\ref{thm:main3} now becomes straightforward, simply by observing that the discrepancy coefficients $\mu_i$ are either all zero or all strictly positive.
	
	\subsection{Organization of the paper}
	
	The paper is organized as follows.  In Section~\ref{sec:gauge}, we construct
	the Euclidean Bianchi gauge on an ALE end and derive the fourth-order elliptic
	equation governing the metric perturbation. This gives the rough homogeneous $|x|^{-2}$ expansion.  In Section~\ref{sec:proof-main-leading}, we analyze the
	remaining coordinate freedom, separate the mass part from the Weyl-type part,
	introduce the Schwarzschild relative gauge, and prove the uniqueness of the
	preferred data; the Burns metric is also computed as an explicit example.
	Finally, in Section~\ref{sec:Kahler-ALE-minimal-resolution}, we specialize the
	construction to scalar-flat K\"ahler ALE metrics on minimal quotient
	resolutions and prove the crepancy criterion for the vanishing of the Weyl tensor at infinity.
	
	\subsection{Acknowledgments}
	
	I am thankful for discussions with my supervisor Yin Hao.
	
	\section{Asymptotic Expansion of the Metric in the Bianchi Gauge}\label{sec:gauge}
	
	In this section, we carry out the first analytic step toward understanding the metric at infinity.  
	Starting from the baseline ALE chart given by Theorem~\ref{thm:baseline-ALE}, we seek a better coordinate system in which the metric perturbation $h=g-g_E$ satisfies a Bianchi-type gauge condition.  
	This gauge is introduced to control the remaining coordinate freedom and to
	put the relevant curvature equations derived from Assumption~\ref{ass:TVbaseline} into forms adapted to elliptic analysis
	on the exterior region. With this preparation, the equation for $h$ becomes a fourth-order system whose principal part is the Euclidean bi-Laplacian, this allows us to obtain a first asymptotic expansion of the metric.  
	The outcome is that the first nontrivial term in the expansion occurs at order $r^{-2}$; this is the first order at which the geometry at infinity will later be seen to carry nontrivial information.
	
	\subsection{Weighted function spaces}\label{sec:weighted}

	Before commencing the formal analysis, we first systematically outline the analytical framework of the entire study. This subsection imports the scale-invariant weighted H\"older spaces used in \cite{WY} and records a small
	toolkit: differentiation, multiplication, composition under small diffeomorphisms, and a bounded right-inverse
	for the Euclidean Laplacian on an exterior domain. These are the crucial analytic inputs needed later.

	We first introduce the weighted H\"older spaces on an exterior domain.
	Fix $R>0$ and set $\Omega_R:=\R^4\setminus B_R$.
	For $r>0$ define the scaling map on annuli
	\begin{align*}
		S_r:\ B_2\setminus B_1 \to B_{2r}\setminus B_r,
		\qquad
		S_r(x)=r x.
	\end{align*}
	For a function $v$ defined on $\Omega_R$, and for $\alpha>4$ and $\beta>0$ with $\beta\notin\N$,
	define
	\begin{align}\label{eq:Xab-def}
		\|v\|_{\mathcal X_{\alpha,\beta}(\Omega_R)}
		:=
		\sup_{r\ge R} r^\beta \|v\circ S_r\|_{C^\alpha(B_2\setminus B_1)}.
	\end{align}
	We denote by $\mathcal X_{\alpha,\beta}(\Omega_R)$ the Banach space of functions with finite norm
	\eqref{eq:Xab-def}. If there is no ambiguity, we write $\mathcal X_{\alpha,\beta}$.
	
	\begin{remark}[Quotients by a finite group]
		When working on an ALE end $(\R^4\setminus B_R)/\Gamma$, we define $\mathcal X_{\alpha,\beta}(\Omega_R/\Gamma)$
		by lifting $\Gamma$-invariant functions to $\Omega_R$ and using \eqref{eq:Xab-def}. Since $\Gamma$ is finite,
		this is equivalent (up to constants depending only on $\Gamma$) to defining the norm using local charts on the quotient.
	\end{remark}
	
	Then we collect some calculus rules used repeatedly: derivatives increase weight by $1$,
	products add weights, and composition with small decaying diffeomorphisms preserves weights.
	
	\begin{lemma}[Derivative estimate]\label{lem:X-derivative}
		For each $i\in\{1,\dots,4\}$, there exists $C=C(\alpha)$ such that
		\begin{align*}
			\|\del_{x_i} v\|_{\mathcal X_{\alpha-1,\beta+1}(\Omega_R)}
			\le
			C\,\|v\|_{\mathcal X_{\alpha,\beta}(\Omega_R)}.
		\end{align*}
	\end{lemma}
	
	\begin{lemma}[Product estimate]\label{lem:X-product}
		Let $f\in\mathcal X_{\alpha_1,\beta_1}(\Omega_R)$ and $g\in\mathcal X_{\alpha_2,\beta_2}(\Omega_R)$.
		Then
		\begin{align*}
			\|fg\|_{\mathcal X_{\min(\alpha_1,\alpha_2),\beta_1+\beta_2}(\Omega_R)}
			\le
			C\,\|f\|_{\mathcal X_{\alpha_1,\beta_1}(\Omega_R)}\,\|g\|_{\mathcal X_{\alpha_2,\beta_2}(\Omega_R)}.
		\end{align*}
	\end{lemma}
	
	\begin{lemma}[Composition under a small decaying diffeomorphism]\label{lem:X-composition}
		Let $f$ be defined on $\Omega_{R/2}$ and satisfy $f=\mathcal O(r^{-\ell})$ for some $\ell>0$.
		Let $u:\Omega_R\to\R^4$ and set $\varphi(x)=x+u(x)$.
		If $\|u\|_{\mathcal X_{\alpha,\gamma}(\Omega_R)}\le 1$ for some $\gamma>0$, then for all sufficiently large $R$,
		\begin{align*}
			\|f\circ\varphi\|_{\mathcal X_{\alpha,\ell}(\Omega_R)}\le C.
		\end{align*}
	\end{lemma}
	
	\begin{remark}
		Lemmas~\ref{lem:X-derivative}--\ref{lem:X-composition} are standard results from \cite{WY} and will be used tacitly. 
		In the sequel, we also assume the following elementary fact without further mention: 
		if $T = \mathcal{O}_\infty(r^{-\beta})$, then $T \in \mathcal{X}_{\alpha,\beta}$ for any $\alpha$, and differentiation improves the weight by one order.
	\end{remark}

	Having established the definition of the basic weighted norms and their algebraic properties, 
	we now state the key analytic input: the existence of a bounded right inverse for $\Delta_E$ on these weighted spaces. 
	This operator, originally constructed in \cite{WY}, will be employed as a black box.
	
	\begin{theorem}\label{thm:YW-poisson}
		There exists a bounded linear map
		\begin{align*}
			\Psi:\ \mathcal X_{\alpha-2,\beta+2}(\Omega_R)\to \mathcal X_{\alpha,\beta}(\Omega_R)
		\end{align*}
		such that for all $f\in \mathcal X_{\alpha-2,\beta+2}(\Omega_R)$,
		\begin{align*}
			\DeltaE\,\Psi(f)=f
			\qquad\text{on }\Omega_R,
		\end{align*}
		and
		\begin{align*}
			\|\Psi(f)\|_{\mathcal X_{\alpha,\beta}(\Omega_R)}
			\le
			C\,\|f\|_{\mathcal X_{\alpha-2,\beta+2}(\Omega_R)},
		\end{align*}
		where $C$ is independent of $R$.
	\end{theorem}
	
	\begin{remark}
		This theorem imposes less restrictive conditions on the parameters $\alpha$ and $\beta$. Specifically, $\alpha > 2$, $\beta > 0$, and $\beta$ is non-integer.
	\end{remark}
	
	\medskip
	
	To address the challenges associated with fourth-order equations in our study, we require a corresponding result concerning the inverse mapping of $\Delta_E^2$. In fact, Theorem~\ref{thm:YW-poisson} stated above admits the following direct corollary:
	
	\begin{corollary}\label{cor:YW-bipoisson}
		Define $\Psi^{(2)}:=\Psi\circ\Psi$. Then $\Psi^{(2)}$ is a bounded linear map
		\begin{align*}
			\Psi^{(2)}:\ \mathcal X_{\alpha-4,\beta+4}(\Omega_R)\to \mathcal X_{\alpha,\beta}(\Omega_R)
		\end{align*}
		such that
		\begin{align*}
			\DeltaE^2\,\Psi^{(2)}(f)=f
			\qquad\text{on }\Omega_R,
		\end{align*}
		and
		\begin{align*}
			\|\Psi^{(2)}(f)\|_{\mathcal X_{\alpha,\beta}(\Omega_R)}
			\le
			C\,\|f\|_{\mathcal X_{\alpha-4,\beta+4}(\Omega_R)}.
		\end{align*}
	\end{corollary}

	\subsection{Construction of the Euclidean Bianchi gauge}\label{sec:bianchi}
	
	We now construct a change of coordinates under which the metric perturbation satisfies the Euclidean Bianchi gauge condition.
	Starting from the baseline ALE chart of Theorem~\ref{thm:baseline-ALE}, write
	\begin{align*}
		g=\gE+h_0
	\end{align*}
	on $\Omega_R$, where for every fixed $\tau\in(1,2)$ one has $h_0\in \mathcal X_{\alpha+1,\tau}(\Omega_R)$ after enlarging $R$ if necessary.
	
	Recall that for a symmetric $2$-tensor $k$ on $\Omega_R$,
	\begin{align*}
		\BE(k):=\divE k-\frac12 d(\trE k).
	\end{align*}
	Our goal is to find a small vector field $u\in \mathcal X_{\alpha+1,\tau-1}(\Omega_R)$ such that the map $\varphi(x)=x+u(x)$ is a diffeomorphism of $\Omega_R$ onto its image, and the pulled-back perturbation $h:=\varphi^*g-\gE$ satisfies the Euclidean Bianchi gauge condition
	\begin{align}\label{eq:target-Bianchi-gauge}
		\BE(h)=0.
	\end{align}

	The key point is that, after linearizing the gauge equation with respect to the coordinate correction $u$, the leading term is simply the Euclidean Laplacian $\DeltaE u$. Indeed, for a smooth vector field $X$,
	\begin{align*}
		\BE(\Lie_X\gE)_i
		&=
		\partial^j(\partial_iX_j+\partial_jX_i)
		-\frac12\partial_i\bigl(2\partial^jX_j\bigr)
		=
		\partial^j\partial_j X_i
		=
		(\DeltaE X)_i.
	\end{align*}
	Thus
	\begin{align*}
		\BE(\Lie_X\gE)=\DeltaE X.
	\end{align*}
	This identity is what makes Theorem~\ref{thm:YW-poisson} directly applicable.
	
	For $u\in \mathcal X_{\alpha+1,\tau-1}(\Omega_R)$ sufficiently small, set $\varphi=\Id+u$ and define
	\begin{align*}
		\mathscr F(u):=\BE(\varphi^*g-\gE).
	\end{align*}
	Our task is to employ the fixed-point principle in $\mathcal X_{\alpha+1,\tau-1}(\Omega_R)$ to find a vector field $u$ that vanishes under the map $\mathscr F$. This procedure is standard. 
	
	In fact, the structure of the gauge equation can be described as follows. For arbitrary fixed $\alpha\in(4,\infty)$ and $\tau\in(1,2)$, there exist $R$ sufficiently large and $\varepsilon>0$ sufficiently small such that for every $u\in \mathcal X_{\alpha+1,\tau-1}(\Omega_R),\|u\|_{\mathcal X_{\alpha+1,\tau-1}(\Omega_R)}<\varepsilon$, the map $\mathscr F(u)$ is well-defined on $\Omega_R$ and admits the expansion
	\begin{align}\label{eq:F-expansion}
		\mathscr F(u)
		=
		\BE(h_0)+\DeltaE u+\mathscr N(u,h_0),
	\end{align}
	where the nonlinear remainder satisfies
	\begin{align}\label{eq:N-bound}
		\|\mathscr N(u,h_0)\|_{\mathcal X_{\alpha-1,\tau+1}(\Omega_R)}
		\le
		C\Bigl(
		\|h_0\|_{\mathcal X_{\alpha+1,\tau}(\Omega_R)}
		\|u\|_{\mathcal X_{\alpha+1,\tau-1}(\Omega_R)}
		+
		\|u\|_{\mathcal X_{\alpha+1,\tau-1}(\Omega_R)}^2
		\Bigr).
	\end{align}
	Moreover, $\mathscr N(\cdot,h_0)$ is locally Lipschitz: for all sufficiently small
	$u_1,u_2\in \mathcal X_{\alpha+1,\tau-1}(\Omega_R)$,
	\begin{align}\label{eq:propN-lip}
		\begin{split}
			&\|\mathscr N(u_1,h_0)-\mathscr N(u_2,h_0)\|_{\mathcal X_{\alpha-1,\tau+1}(\Omega_R)}
			\\
			&\qquad \le
			C\Bigl(
			\|h_0\|_{\mathcal X_{\alpha+1,\tau}(\Omega_R)}
			+\|u_1\|_{\mathcal X_{\alpha+1,\tau-1}(\Omega_R)}
			+\|u_2\|_{\mathcal X_{\alpha+1,\tau-1}(\Omega_R)}
			\Bigr)
			\|u_1-u_2\|_{\mathcal X_{\alpha+1,\tau-1}(\Omega_R)}.
		\end{split}
	\end{align}
	
	With the gauge equation prepared, the existence of a Bianchi gauge follows from the fixed-point principle. 
	
	More precisely, according to the expansion \eqref{eq:F-expansion}, solving $\BE(\varphi^*g-\gE)=0$ on $\Omega_R$ is equivalent to finding a fixed point of
	\begin{align*}
		\mathcal T(u):=-\Psi\bigl(\BE(h_0)+\mathscr N(u,h_0)\bigr),
	\end{align*}
	where $\Psi$ is the bounded right inverse of $\DeltaE$ given by Theorem~\ref{thm:YW-poisson}. Thanks to the estimates \eqref{eq:N-bound} and \eqref{eq:propN-lip}, Banach's fixed-point theorem yields a unique
	$u\in \mathcal X_{\alpha+1,\tau-1}(\Omega_R)$ such that $u=\mathcal T(u)$. Applying $\DeltaE$ to the fixed-point identity, we derive a solution to the Bianchi gauge condition \eqref{eq:target-Bianchi-gauge}. To summarize,
	
	\begin{proposition}\label{prop:solve-bianchi}
		Fix $\tau\in(1,2)$ and a non-integer $\alpha\in(4,\infty)\setminus\N$.
		There exists $R$ sufficiently large and a vector field
		$u\in\mathcal X_{\alpha+1,\tau-1}(\Omega_R)$ such that $\varphi(x)=x+u(x)$
		is a diffeomorphism of $\Omega_R$ (onto its image) and satisfies
		\begin{align*}
			\BE(\varphi^*g-\gE)=0.
		\end{align*}
		Moreover,
		\begin{align*}
			\|u\|_{\mathcal X_{\alpha+1,\tau-1}(\Omega_R)}
			\le
			C\,\|h_0\|_{\mathcal X_{\alpha,\tau}(\Omega_R)},
		\end{align*}
		where $C$ depends only on $\alpha,\tau$ and the fixed baseline chart.
	\end{proposition}
	
	\medskip
	
	Proposition~\ref{prop:solve-bianchi} provides the first analytic normalization of the ALE end: after a small asymptotically trivial change of coordinates, the perturbation $h=g-\gE$ satisfies the Euclidean Bianchi condition. This is the correct starting point for the fourth-order analysis of the curvature equations derived from Assumption~\ref{ass:TVbaseline}. In the next subsection, we will use this gauge to derive the fourth-order
	equation satisfied by the metric perturbation on the exterior region.
	
	\subsection{The fourth-order equation in Euclidean Bianchi gauge}\label{subsec:new bach-flat equation}
	
	We now continue from the Euclidean Bianchi gauge constructed in the previous subsection. Our goal here is to show that the two alternatives in Assumption~\ref{ass:TVbaseline} lead, in these coordinates, to a fourth-order equation whose principal part is $\DeltaE^2 h$. This is the form required for the exterior biharmonic analysis in the subsequent discussion. We will use the schematic notation $T_1*\cdots*T_N$ for a finite linear combination of Euclidean contractions of the tensors $T_1,\dots,T_N$ with universal coefficients.
	
	Let $h:=\varphi^*\Phi^*g-\gE$ be the perturbation produced by Proposition~\ref{prop:solve-bianchi}, then $\BE(h)=0$ holds on $\Omega_R$. So in Euclidean coordinates,
	\begin{align}
		\partial^k h_{ik}-\frac12\,\partial_i(\trE h)=0.
		\label{eq:euclidean-bianchi-gauge-h}
	\end{align}
	We shall use \eqref{eq:euclidean-bianchi-gauge-h} repeatedly below. Moreover, $h$ is still a small decaying perturbation of $\gE$ on the exterior region. Notice that under Assumption~\ref{ass:TVbaseline}, we have $\Scal(g)=0$, and that either $g$ is Bach-flat or $g$ has harmonic curvature. 
	
	In the Bach-flat case, by the naturality of the Bach tensor, we conclude that $\Bach(\gE+h)$ vanishes on $\Omega_R$. For the four-dimensional formula
	\begin{align*}
		\Bach_{ij}
		=
		\nabla^k C_{kij}+P^{k\ell}W_{ikj\ell},
	\end{align*}
	where $W$ is the Weyl tensor, $P$ is the Schouten tensor, and $C$ is the Cotton tensor, linearizing the term $P^{k\ell}W_{ikj\ell}$ at $\gE$ and using the product rule gives
	\begin{align*}
		\bigl(D(P^{k\ell}W_{ikj\ell})\bigr)_{\gE}(h)
		=
		(DP)_{\gE}(h)^{k\ell}W(\gE)_{ikj\ell}
		+
		P(\gE)^{k\ell}(DW)_{\gE}(h)_{ikj\ell}=0,
	\end{align*}
	since $\gE$ is flat. Therefore, in the Bach-flat case, the equation
	$\Bach(\gE+h)=0$ can be written in the form
	\begin{align}\label{eq:Bach-flat-equation}
		\nabla^k C_{kij}
		=
		-
		P^{k\ell}W_{ikj\ell},
	\end{align}
	and the right-hand side has no linear term at $\gE$. Thus it is absorbed into
	the nonlinear remainder when we extract the linearized fourth-order operator.
	
	In the harmonic-curvature case, $\divergence\Rm=0$ is equivalent to $\nabla_k\Ric_{ij}=\nabla_j\Ric_{ik}$. Since $\Scal(g)=0$, the Schouten tensor is $P=\frac12\Ric$, and hence the Cotton tensor $C_{kij}$, defined as $\nabla_kP_{ij}-\nabla_jP_{ik}$, vanishes. In particular, $\nabla^kC_{kij}=0$ on $\Omega_R$. Thus in both cases we may derive the perturbed bi-harmonic equation for $h$ via linearizing the tensor $\nabla^kC_{kij}$ to obtain the fourth-order linear principal part.
	
	Under the gauge condition \eqref{eq:euclidean-bianchi-gauge-h}, the standard linearized curvature formulas reduce to
	\begin{align*}
		(\delta\Ric)_{ij}=-\frac12\,\DeltaE h_{ij},
		\quad
		\delta\Scal=-\frac12\,\DeltaE(\trE h).
	\end{align*}
	Hence the linearized Schouten tensor is
	\begin{align}\label{eq:linearized-schouten}
		(\delta P)_{ij}
		=
		\frac12\Bigl((\delta\Ric)_{ij}-\frac16\,\delta\Scal\,\delta_{ij}\Bigr)
		=
		-\frac14\,\DeltaE h_{ij}
		+
		\frac1{24}\,\delta_{ij}\,\DeltaE(\trE h).
	\end{align}
	To pass from $\delta P$ to the linearized Cotton-divergence term, we linearize the Cotton tensor. Since $P(\gE)=0$, all connection-variation terms are multiplied by the
	background Schouten tensor and hence vanish. Therefore
	$(\delta C)_{kij}$ reduces to
	$\partial_k(\delta P)_{ij}-\partial_j(\delta P)_{ik}$.
	Since also $C(\gE)=0$, linearizing the divergence gives
	\begin{align}\label{eq:linearized-bach}
		\bigl(D(\nabla^kC_{kij})\bigr)_{\gE}(h)
		=
		\partial^k(\delta C)_{kij}
		=
		\DeltaE(\delta P)_{ij}-\partial_j\bigl(\partial^k(\delta P)_{ik}\bigr).
	\end{align}
	As for the divergence of $\delta P$, combining formula \eqref{eq:linearized-schouten} for $(\delta P)_{ij}$ with the Bianchi identity \eqref{eq:euclidean-bianchi-gauge-h} yields
	\begin{align*}
		\partial^k(\delta P)_{ik}
		=
		-\frac1{12}\,\partial_i\DeltaE(\trE h),
	\end{align*}
	substituting this into the preceding identity \eqref{eq:linearized-bach} yields
	\begin{align}
		\bigl(D(\nabla^kC_{kij})\bigr)_{\gE}(h)
		=
		-\frac14\,\DeltaE^2 h_{ij}
		+
		\frac1{24}\,\delta_{ij}\,\DeltaE^2(\trE h)
		+
		\frac1{12}\,\partial_i\partial_j\DeltaE(\trE h).
		\label{eq:DBach-explicit-bianchi-short}
	\end{align}
	
	To handle the trace terms, we use the scalar-flat condition. Since $\Scal(\gE+h)$ vanishes and the linearization of the scalar curvature in Bianchi gauge is $\delta\Scal=-\frac12\,\DeltaE(\trE h)$, the remaining terms in the scalar-curvature expansion are at least quadratic. Hence
	\begin{align}
		\DeltaE(\trE h)
		=
		h*\partial^2 h+\partial h*\partial h.
		\label{eq:trace-controlled-schematic-short}
	\end{align}
	Differentiating twice, we obtain only quadratic terms, namely combinations of $h*\partial^4 h$, $\partial h*\partial^3 h$, and $\partial^2 h*\partial^2 h$. In particular, the trace terms in \eqref{eq:DBach-explicit-bianchi-short} do not produce any new linear fourth-order term, and may therefore be included in the nonlinear error.
	
	Finally, we expand the corresponding nonlinear equation around $\gE$.
	In the harmonic-curvature case we use the consequence
	$\nabla^k C_{kij}=0$; in the Bach-flat case we use \eqref{eq:Bach-flat-equation}. The two
	equations have the same linear fourth-order part, namely
	$D(\nabla^k C_{kij})_{\gE}(h)$.
	Using \eqref{eq:DBach-explicit-bianchi-short} and then applying
	\eqref{eq:trace-controlled-schematic-short} to absorb the trace terms
	into the nonlinear part, we obtain
	\begin{align*}
		\frac14\,\DeltaE^2 h
		+
		h*\partial^4 h
		+
		\partial h*\partial^3 h
		+
		\partial^2 h*\partial^2 h
		+
		\partial h*\partial h*\partial^2 h
		+
		\partial h*\partial h*\partial h*\partial h
		=
		0
		\qquad\text{on }\Omega_R.
	\end{align*}
	Thus, in Euclidean Bianchi gauge, both the scalar-flat Bach-flat equation and the scalar-flat harmonic-curvature equation are small quasilinear perturbations of the Euclidean biharmonic equation, with principal part $\DeltaE^2 h$.
	
	\subsection{Biharmonic reduction and rough metric expansion}
	\label{subsec:proof-leading-step2}
	
	In the previous subsection, we showed that, in Euclidean Bianchi gauge, the
	equations arising from both alternatives in Assumption~\ref{ass:TVbaseline}
	take the common schematic form
	\begin{align*}
		\frac14\,\DeltaE^2 h=\mathcal N(h).
	\end{align*}
	where
	\begin{align*}
		\mathcal N(h)
		:=
		h*\partial^4 h
		+
		\partial h*\partial^3 h
		+
		\partial^2 h*\partial^2 h
		+
		\partial h*\partial h*\partial^2 h
		+
		(\partial h)^{*4}.
	\end{align*}
	We now use this equation to extract the first term in the expansion of $h$. The basic idea is to subtract a particular solution of the bi-Poisson equation and reduce the problem to a biharmonic one.
	
	From the baseline ALE decay and the gauge-fixing step, we may choose a non-integer
	$\alpha>4$ and $\tau_0\in(3/2,2)$ such that $h\in \mathcal X_{\alpha,\tau_0}(\Omega_R)$. Since every term in $\mathcal N(h)$ contains at least two factors of $h$ or its derivatives, it follows directly that $\mathcal N(h)\in \mathcal X_{\alpha-4,\,2\tau_0+4}(\Omega_R)$. Because $2\tau_0\in(3,4)$, the weight $2\tau_0$ is admissible for Corollary~\ref{cor:YW-bipoisson}. We therefore take $v:=\Psi^{(2)}\bigl(4\,\mathcal N(h)\bigr)\in \mathcal X_{\alpha,\,2\tau_0}(\Omega_R)$ and define $w:=h-v$. Then
	\begin{align*}
		\DeltaE^2 v=4\,\mathcal N(h),
		\qquad
		\DeltaE^2 w=0
		\qquad\text{on }\Omega_R,
	\end{align*}
	Since $\tau_0>1$, both $h$ and $v$ are $o(r^{-1})$, and therefore so is $w$.
	
	We now invoke the standard asymptotic expansion for biharmonic functions on an exterior domain; see, for instance, \cite{Kalf}. Applied componentwise to the $\Gamma$-invariant tensor $w$, it gives an expansion in spherical harmonics on $S^3/\Gamma$. Under the decay condition $w=o(r^{-1})$, all growing modes, the constant and logarithmic modes in the $\ell=0$ channel, and the $r^{-1}$ mode in the $\ell=1$ channel are excluded. Hence the first possible nonzero term occurs at order $r^{-2}$, and only the $\ell=0$ and $\ell=2$ modes can appear at that order. It follows that there exists a smooth $\Gamma$-invariant spherical tensor $H$ on $S^3$ such that
	\begin{align}\label{eq:w-standard-biharmonic-expansion}
		w(x)
		=
		\frac1{r^2}H(\theta)+\mathcal O_\infty(r^{-3}),
		\qquad
		r=|x|,
		\quad
		\theta=\frac{x}{r}.
	\end{align}
	Meanwhile, $v\in \mathcal X_{\alpha,\,2\tau_0}(\Omega_R)$ implies $v=\mathcal O_\infty(r^{-3})$, therefore, combining $h=w+v$ with \eqref{eq:w-standard-biharmonic-expansion}, we obtain
	\begin{align*}
		h(x)
		=
		\frac{1}{r^2}H(\theta)+\mathcal O_\infty(r^{-3}).
	\end{align*}
	
	Furthermore, since $H$ only contains the $\ell=0$ and $\ell=2$ modes, the tensor $r^{-2}H(\theta)$ can be written in Cartesian form as
	\begin{align*}
		\frac1{r^2}H(\theta)
		=
		A_{ijk\ell}\frac{x_kx_\ell}{r^4},
	\end{align*}
	where the constant tensor $A=(A_{ijk\ell})$ is defined as
	\begin{align}\label{eq:A-def}
		A_{ijk\ell}
		:=
		(D_0)_{ij}\delta_{k\ell}
		+
		\sum_{a\in\mathcal A_2}(C_{2,a})_{ij}Q^{(a)}_{k\ell},
	\end{align}
	here $\mathcal A_2$ is an index set for a basis $\{Y_{2,a}\}_{a\in\mathcal A_2}$ of the $\Gamma$-invariant $\ell=2$ eigenspace on $S^3/\Gamma$; each $Y_{2,a}$ can be expressed as $Y_{2,a}(\theta)=Q^{(a)}_{k\ell}\theta_k\theta_\ell$ for some $Q^{(a)}\in S^2_0(\R^4)$, and $D_0$ and $C_{2,a}$ are constant symmetric $2$-tensors.
	
	We have therefore proved the following proposition.
	
	\begin{proposition}\label{prop:proof-leading-h-expansion}
		There exist $R\gg1$ and a constant tensor $A=(A_{ijk\ell})$ such that
		\begin{align}\label{eq:h-leading-cartesian}
			h_{ij}(x)
			=
			A_{ijk\ell}\,\frac{x_kx_\ell}{r^4}
			+
			\mathcal O_\infty(r^{-3})
			\qquad\text{on }\Omega_R.
		\end{align}
	\end{proposition}
	
	\begin{remark}
		For convenience, we shall refer to \eqref{eq:h-leading-cartesian} as \textit{the rough expansion} of the metric in the sequel.
	\end{remark}
	
	\medskip
	
	We now examine the coefficient $A$ in the rough expansion \eqref{eq:h-leading-cartesian}. 
	
	Write $(h^A)_{ij}(x):=A_{ijk\ell}\,\frac{x_kx_\ell}{r^4}$, then the Euclidean Bianchi gauge $\BE(h)=0$ imposes a condition on the leading term $h^A$. Directly computation reads
	\begin{align*}
		(\BE(h^A))_i
		=
		\frac{1}{r^3}\Bigl( -4\,A_{ijk\ell}\,\theta_j\theta_k\theta_\ell
		+
		\bigl(2A_{ijj\ell}-A_{pp i\ell}\bigr)\theta_\ell
		+
		2A_{ppk\ell}\,\theta_i\theta_k\theta_\ell\Bigr),
	\end{align*}
	where $\theta=\frac{x}{r}$. Since $h=h^A+\mathcal O_\infty(r^{-3})$, we have
	\begin{align*}
		\BE(h)=\BE(h^A)+\mathcal O_\infty(r^{-4}),
	\end{align*}
	and therefore $\BE(h)=0$ forces
	\begin{align}\label{eq:rougn-bianchi-gauge}
		-4\,A_{ijk\ell}\,\theta_j\theta_k\theta_\ell
		+
		\bigl(2A_{ijj\ell}-A_{pp i\ell}\bigr)\theta_\ell
		+
		2A_{ppk\ell}\,\theta_i\theta_k\theta_\ell\equiv0
		\qquad\text{on }\Sph^3.
	\end{align}
	
	It is convenient to rewrite this condition in tensor form. Set $U_{k\ell}:=A_{ppk\ell}$ and $V_{i\ell}:=A_{ijj\ell}$, since only the full symmetrization of $A$ in $j,k,\ell$ contributes to
	$A_{ijk\ell}\theta_j\theta_k\theta_\ell$, and $|\theta|=1$, the identity \eqref{eq:rougn-bianchi-gauge} is equivalent to
	\begin{align}\label{eq:bianchi-tensor-form}
		\begin{split}
			A_{i(jk\ell)}
			&=
			\frac{1}{12}\Bigl(
			\delta_{jk}(2V_{i\ell}-U_{i\ell})
			+
			\delta_{j\ell}(2V_{ik}-U_{ik})
			+
			\delta_{k\ell}(2V_{ij}-U_{ij})
			\Bigr)
			\\
			&\qquad +
			\frac{1}{6}\Bigl(
			\delta_{ij}U_{k\ell}
			+
			\delta_{ik}U_{j\ell}
			+
			\delta_{i\ell}U_{jk}
			\Bigr),
		\end{split}
	\end{align}
	where $A_{i(jk\ell)}$ denotes the full symmetrization in $j,k,\ell$. Contracting
	\eqref{eq:bianchi-tensor-form} with $\delta_{jk}$ further gives
	\begin{align}\label{eq:bianchi-trace-identity}
		2A_{ijj\ell}
		-
		2A_{i\ell kk}
		-
		A_{pp i\ell}
		+
		\delta_{i\ell}A_{ppkk}
		=
		0.
	\end{align}
	
	We summarize the outcome as follows.
	
	\begin{proposition}\label{prop:rough-A-structure}
		Let $A$ be the tensor in Proposition~\ref{prop:proof-leading-h-expansion}. Then:
		\begin{enumerate}[label=\textup{(\arabic*)}]
			\item $A$ satisfies the pair symmetries
			\begin{align*}
				A_{ijk\ell}=A_{jik\ell}=A_{ij\ell k}.
			\end{align*}
			\item $A$ satisfies the tensor identity \eqref{eq:bianchi-tensor-form}.
			\item In particular, the partial traces of $A$ satisfy \eqref{eq:bianchi-trace-identity}.
		\end{enumerate}
	\end{proposition}
	
	\begin{proof}
		The pair symmetries follow directly from the definition \eqref{eq:A-def}. The remaining statements are exactly the identities established above.
	\end{proof}
	
	\medskip
	
	So far, we have established a rough asymptotic expansion for the metric in the Euclidean Bianchi gauge, namely \eqref{eq:h-leading-cartesian}. Furthermore, the leading coefficient $A_{ijk\ell}$ satisfies a set of constraints, as detailed in Proposition~\ref{prop:rough-A-structure}. This gives a concrete leading term for the metric at infinity. In the next section, we will refine this rough $r^{-2}$ order term and isolate the part of $A$ that carries genuine geometric information.

	\section{Identification of Curvature at Infinity}\label{sec:proof-main-leading}
	
	This section proves Theorem~\ref{thm:main1} and Theorem~\ref{thm:main2}. 
	We start from the rough $r^{-2}$ coefficient $A$ obtained in the Euclidean Bianchi gauge in the previous section and analyze its remaining coordinate ambiguity as well as the potential intrinsic geometric information it contains.
	
	Motivated by the work of Wang--Yin~\cite{WY}, we consider whether the leading coefficient $A$ can be further simplified by a coordinate change so that it has the algebraic form expected of a Weyl-type term. Specifically, we impose the following coordinate transformation on the current coordinate system:
	\begin{align}\label{eq:corrdinate-traneformation}
		x=\varphi(y)=y+u(y),
	\end{align}
	where $u$ is chosen to be suitable for our purposes. We aim to track the Weyl-type information from the leading coefficient $A$ by the coordinate transformation \eqref{eq:corrdinate-traneformation}, naturally, we proceed by targeting the properties satisfied by the Weyl-type tensor.
	
	A canonical target, which is precisely the one considered in \cite{WY}, can be described as follows.
	
	Let us begin with some notation. Denote the space of algebraic curvature tensors on $\R^4$ by $\mathcal C$, and the algebraic Weyl space by $\mathcal W\subset \mathcal C$. Let
	\begin{align}\label{eq:def-s-map}
		s:\mathcal C\longrightarrow S^2(\R^4)\otimes S^2(\R^4),
		\quad
		(s(R))_{ijk\ell}:=\frac12\bigl(R_{ik\ell j}+R_{i\ell kj}\bigr).
	\end{align}
	Write $\widetilde{\mathcal C}:=s(\mathcal C)$ and $\widetilde{\mathcal W}:=s(\mathcal W)$. In fact, the so-called Weyl-type tensor that we aim to track from $A$ is precisely an element of $\widetilde{\mathcal W}$. Therefore, we need to gain a clear understanding of the elements in $\widetilde{\mathcal W}$.
	
	The following position reveals a useful criterion for a tensor to belong to $\widetilde{\mathcal W}$.
	
	\begin{proposition}\label{prop:Wtilde-eq}
		Let $A\in S^2(\R^4)\otimes S^2(\R^4)$. Then $A\in \widetilde{\mathcal W}$ if and only if
		$A$ satisfies the following three conditions:
		\begin{enumerate}[label=\textup{(\roman*)}]
			\item $A_{ijk\ell}=A_{jik\ell}=A_{ij\ell k}$;
			\item $A_{ppk\ell}=0$ and $A_{ijj\ell}=0$;
			\item
			$A_{ijk\ell}+A_{ik\ell j}+A_{i\ell jk}=0.$
		\end{enumerate}
		Moreover, if these conditions hold, then $A=s(\mathcal W_\infty)$ for a unique
		$\mathcal W_\infty\in\mathcal W$, and in fact
		\begin{align}\label{eq:def-R-map}
			(\mathbf R A)_{ijk\ell}:=A_{j\ell ik}-A_{jk i\ell}-A_{i\ell jk}+A_{ik j\ell},
		\end{align}
		satisfies
		\begin{align}\label{eq:inverse-formula}
			\mathcal W_\infty=-\frac13\,\mathbf R A.
		\end{align}
	\end{proposition}
	
	This proposition is essentially a standard algebraic fact. We only record a brief proof.
	
	\begin{proof}
		The implication $A=s(\mathcal W_\infty)\Rightarrow$\textup{(i)--(iii)} is immediate from the algebraic symmetries, trace-freeness, and first Bianchi identity of $\mathcal W_\infty$.
		
		Conversely, assume that $A$ satisfies \textup{(i)--(iii)}. A routine use of \textup{(iii)} together with the pairwise symmetries in \textup{(i)} shows that $A$ also satisfies the pair-exchange symmetry $	A_{ijk\ell}=A_{k\ell ij}$. Now define $R:=\mathbf RA$ by \eqref{eq:def-R-map}. Using \textup{(i)}, \textup{(iii)}, and the pair-exchange symmetry, one checks directly that $R$ is an algebraic curvature tensor. Moreover, \textup{(ii)} implies that its Ricci contraction vanishes, while Ricci-flatness is equivalent to
		total trace-freeness, so in fact $R\in\mathcal W$.
		
		Finally, substituting \eqref{eq:def-R-map} into the definition of $s$ and simplifying by \textup{(i)} and \textup{(iii)}, one obtains
		\begin{align*}
			s(\mathbf RA)=-3A,
		\end{align*}
		which proves that $A\in\widetilde{\mathcal W}$ and gives the formula \eqref{eq:inverse-formula}.
		
		Uniqueness follows from the same identity: if $A=s(\mathcal W_1)=s(\mathcal W_2)$, then
		\begin{align*}
			-3\mathcal W_1=\mathbf R(s(\mathcal W_1))=\mathbf R(s(\mathcal W_2))=-3\mathcal W_2.
		\end{align*}
	\end{proof}
	
	\begin{remark}
		For notational convenience, we shall refer to the three properties summarized in the proposition as: pair symmetry, partial trace-free, and first Bianchi identity.
	\end{remark}
	
	\medskip
	
	Returning to our previous discussion, since we aim to trace the Weyl-type tensor within the leading coefficient $A$, it is natural to track the three properties satisfied by the Weyl-type tensor introduced in Proposition~\ref{prop:Wtilde-eq}. Of course, we cannot expect $A$ to satisfy these properties at this stage. However, starting from the coordinate transformation \eqref{eq:corrdinate-traneformation}, we can eliminate the coordinate freedom to track these properties. If a certain property is not satisfied, we can remove the obstruction so that the remaining part satisfies the corresponding property, allowing us to proceed to the next one. This is the overall philosophy of this section.
	
	\subsection{First elimination attempt and the partial trace-Free obstruction}
	\label{subsec:partial-trace-free-obstruction S}
	
	We continue from the rough expansion in Euclidean Bianchi gauge
	\begin{align}\label{eq:h-leading-cartesian2}
		h_{ij}(x)
		=
		A_{ijk\ell}\,\frac{x_kx_\ell}{r^4}
		+
		\mathcal O_\infty(r^{-3}),
		\qquad r=|x|,
	\end{align}
	obtained in Subsection~\ref{subsec:proof-leading-step2}. The coefficient $A$ still depends on the choice of exterior
	coordinates. We must remove as much of this ambiguity as possible.
	
	We now formally begin tracking the three properties from $A$ mentioned in Proposition~\ref{prop:Wtilde-eq}. First, the pair symmetry is evident; indeed, it is guaranteed by Proposition~\ref{prop:rough-A-structure}. Let us now turn our attention to the second property, namely the partial trace-free condition.
	
	As previously mentioned, our first tool for tracking the Weyl-type tensor from the leading coefficient $A$ is the coordinate transformation of the form \eqref{eq:corrdinate-traneformation}. Therefore, let us first examine how $A$ transforms under the change of coordinates \eqref{eq:corrdinate-traneformation}.
	
	Denote the metric after the transformation of coordinates as $\tilde g=\gE+\tilde h$, corresponding to the original $g=\gE+h$. If we require $u$ in the coordinate transformation \eqref{eq:corrdinate-traneformation} to satisfy $u=\mathcal O_\infty(r^{-1})$, and combine this with $h=\mathcal O_\infty(r^{-2})$, we obtain the following expansion for the first non-trivial term $\tilde h$:
	\begin{align}\label{eq:pullback-linearization-r-2}
		\tilde h_{ij}(y)
		=
		h_{ij}(y)+\partial_i u_j+\partial_j u_i+\mathcal O_\infty(r^{-3}).
	\end{align}
	Thus, at the level of the leading coefficient, the pullback is governed by the linear term $\mathcal L_u\gE$.
	
	We further specialize the function $u$ in the coordinate transformation \eqref{eq:corrdinate-traneformation} to be $u_i(y)=B_{ij}\tfrac{y_j}{|y|^2}$, where $B$ is a constant matrix, then the linear term $\mathcal L_u\gE$ in \eqref{eq:pullback-linearization-r-2} can be computed as
	\begin{align*}
		(\mathcal L_u\gE)_{ij}(y)
		=
		\Psi_4(B)_{ijk\ell}\,\frac{y_k y_\ell}{r^4},
	\end{align*}
	where
	\begin{align}\label{eq:Psi4-def}
		\Psi_4(B)_{ijk\ell}
		:=
		(B_{ij}+B_{ji})\,\delta_{k\ell}
		-
		\Bigl(
		B_{ik}\,\delta_{j\ell}
		+
		B_{i\ell}\,\delta_{jk}
		+
		B_{jk}\,\delta_{i\ell}
		+
		B_{j\ell}\,\delta_{ik}
		\Bigr).
	\end{align}
	
	Therefore, corresponding to the asymptotic expansion \eqref{eq:h-leading-cartesian2} of the first non-trivial term before the coordinate transformation, under $x_i=y_i+B_{ij}\tfrac{y_j}{|y|^2}$, $h$ transforms into
	\begin{align*}
		\tilde h_{ij}(y)
		=
		\Bigl(A_{ijk\ell}+\Psi_4(B)_{ijk\ell}\Bigr)\frac{y_k y_\ell}{r^4}
		+
		\mathcal O_\infty(r^{-3}),
	\end{align*}
	so the leading coefficient transforms by
	\begin{align*}
		A\longmapsto \widetilde A:=A+\Psi_4(B).
	\end{align*}
	After averaging over $\Gamma$, we may assume that $B$ is $\Gamma$-invariant.
	
	The next lemma records the effect of $\Psi_4(B)$ on the two partial traces of $A$. This is a direct contraction of \eqref{eq:Psi4-def}.
	
	\begin{lemma}\label{lem:Psi4-trace-identities}
		Let $m=4$ and let $\Psi_4(B)$ be defined by \eqref{eq:Psi4-def}. Then
		\begin{align*}
			\begin{split}
				\Psi_4(B)_{ijkk}
				&=
				2\,(B_{ij}+B_{ji}),\\
				\Psi_4(B)_{ppk\ell}
				&=
				2\,\tr(B)\,\delta_{k\ell}-4\,B_{(k\ell)},\\
				\Psi_4(B)_{ijj\ell}
				&=
				-4\,B_{(i\ell)}-6\,B_{[i\ell]}-\tr(B)\,\delta_{i\ell},
			\end{split}
		\end{align*}
		where $B_{(i\ell)}=\frac12(B_{i\ell}+B_{\ell i})$ and
		$B_{[i\ell]}=\frac12(B_{i\ell}-B_{\ell i})$.
	\end{lemma}
	
	\medskip
	
	Now let $U_{k\ell}:=A_{ppk\ell}$ and $V_{i\ell}:=A_{ijj\ell}$. If one wants $\widetilde A$ to be trace-free in both partial traces, then one must solve
	\begin{align*}
		\widetilde A_{ppk\ell}=0,
		\qquad
		\widetilde A_{ijj\ell}=0.
	\end{align*}
	The next lemma shows that this system is not always solvable.
	
	\begin{lemma}\label{lem:LA-Psi4-trace-normalization}
		Assume that $A$ satisfies the pair symmetries, i.e. $A_{ijk\ell}=A_{jik\ell}=A_{ij\ell k}$. Then there exists a matrix $B$ such that, with $W:=A+\Psi_4(B)$, there hold $W_{ppk\ell}=0$ and $W_{ijj\ell}=0$, if and only if
		\begin{align}\label{eq:first-obstruction}
			\bigl(V_{(i\ell)}-U_{i\ell}\bigr)_0=0
			\quad\text{and}\quad
			\tr(V)=-2\,\tr(U).
		\end{align}
	\end{lemma}
	
	\begin{proof}
		By Lemma~\ref{lem:Psi4-trace-identities}, the condition $W_{ppk\ell}=0, W_{ijj\ell}=0$ is equivalent to
		\begin{subequations}\label{eq:partial-trace}
			\begin{empheq}[left=\empheqlbrace]{align}
				0 &= U_{k\ell} + 2\,\tr(B)\,\delta_{k\ell} - 4\,B_{(k\ell)}, \label{eq:partial-trace-a} \\
				0 &= V_{i\ell} - 4\,B_{(i\ell)} - 6\,B_{[i\ell]} - \tr(B)\,\delta_{i\ell}. \label{eq:partial-trace-b}
			\end{empheq}
		\end{subequations}
		
		We first prove the necessity of \eqref{eq:first-obstruction}. Assume that equations \eqref{eq:partial-trace} are solvable. Starting from \eqref{eq:partial-trace}, we extract relevant information to deduce \eqref{eq:first-obstruction}. Taking traces of both equations in \eqref{eq:partial-trace}, we obtain
		$\tr(B)=-\frac14\tr(U)$ and $\tr(B)=\frac18\tr(V)$, hence $\tr(V)=-2\tr(U)$. 
		
		The verification of $(V_{(i\ell)} - U_{i\ell})_0 = 0$ is also straightforward. Equation \eqref{eq:partial-trace-a} yields $B_{(k\ell)}=\frac14\,U_{k\ell}+\frac12\,\tr(B)\,\delta_{k\ell}$.
		Substituting this into the symmetric part of equation \eqref{eq:partial-trace-b}, we get
		\begin{align*}
			0
			&=
			V_{(i\ell)}-4\,B_{(i\ell)}-\tr(B)\,\delta_{i\ell} \\
			&=
			V_{(i\ell)}-U_{i\ell}-3\,\tr(B)\,\delta_{i\ell}.
		\end{align*}
		Therefore $V_{(i\ell)}-U_{i\ell}$ is pure trace, that is,
		$\bigl(V_{(i\ell)}-U_{i\ell}\bigr)_0=0$.
		This proves the necessity of \eqref{eq:first-obstruction}.
		
		Conversely, assume that \eqref{eq:first-obstruction} holds. Since
		$\bigl(V_{(i\ell)}-U_{i\ell}\bigr)_0=0$, the tensor $V_{(i\ell)}-U_{i\ell}$ is pure trace.
		Its trace equals $\tr(V)-\tr(U)=-3\tr(U)$ by the second condition in
		\eqref{eq:first-obstruction}. Hence
		\begin{align*}
			V_{(i\ell)}-U_{i\ell}
			=
			-\frac34\,\tr(U)\,\delta_{i\ell}.
		\end{align*}
		
		We now define $B$ by setting
		$B_{[i\ell]}:=\frac16\,V_{[i\ell]}$,
		$\tr(B):=-\frac14\,\tr(U)$, and
		$B_{(k\ell)}:=\frac14\,U_{k\ell}-\frac18\,\tr(U)\,\delta_{k\ell}$.
		Then equation \eqref{eq:partial-trace-a} holds by construction, and the skew-symmetric
		part of equation \eqref{eq:partial-trace-b} also holds by the definition of $B_{[i\ell]}$.
		For the symmetric part, we compute
		\begin{align*}
			V_{(i\ell)}-4\,B_{(i\ell)}-\tr(B)\,\delta_{i\ell}
			=
			V_{(i\ell)}-U_{i\ell}+\frac34\,\tr(U)\,\delta_{i\ell}
			=
			0.
		\end{align*}
		Thus \eqref{eq:partial-trace} is solvable, and therefore $W_{ppk\ell}=0$ and $W_{ijj\ell}=0$ follow.
	\end{proof}
	
	\medskip
	The lemma shows that the trace-killing problem is not always solvable, and the failure of the two equalities in \eqref{eq:first-obstruction} constitutes a potential obstruction. This naturally leads us to analyze whether the tensor $\bigl(V_{(i\ell)}-U_{i\ell}\bigr)_0$ can indeed vanish under the current conditions, and whether the trace relation $\tr(V)=-2\,\tr(U)$ holds.
	
	In fact, under the current setting, particularly based on the previously constructed Euclidean Bianchi gauge, the second trace relation is problematic. This is because directly contracting \eqref{eq:bianchi-trace-identity} yields $\tr(U)=-2\tr(V)$, which clearly does not directly imply the trace relation required by \eqref{eq:first-obstruction}. Fortunately, however, we are able to prove that $\bigl(V_{(i\ell)}-U_{i\ell}\bigr)_0=0$ is indeed valid.
	
	\begin{proposition}\label{prop:residual-charge-vanishes}
		Assume that on $\Omega_R\subset \R^4/\Gamma$ we have a metric $g=\gE+h$ in Euclidean
		Bianchi gauge $\BE(h)=0$ and with $\Scal(g)=0$. Assume moreover that $h=h^A+\mathcal O_\infty(r^{-3})$, where $(h^A)_{ij}(x)=A_{ijk\ell}\tfrac{x_kx_\ell}{r^4}$. Write $U_{k\ell}:=A_{ppk\ell}$ and $V_{i\ell}:=A_{ijj\ell}$, then
		\begin{align}\label{eq:vanish-result}
			\bigl(V_{(i\ell)}-U_{i\ell}\bigr)_0=0.
		\end{align}
	\end{proposition}
	
	\begin{proof}
		Recall that in the Euclidean Bianchi gauge, by carrying out the Euclidean linearization of the scalar curvature at $\gE$, we obtain the following expansion:
		\begin{align*}
			\Scal(g)=-\frac12\,\DeltaE\trE(h)+\mathcal O\bigl(|h|\,|\partial^2 h|+|\partial h|^2\bigr),
		\end{align*}
		more precisely, the remainder term is of order $\mathcal O(r^{-6})$.
		
		We first prove $U$ is pure-trace.
		
		Denote $h=h^A+\eta$ and the Euclidean linearization $\Scal^\lin(\eta)=-\frac12\DeltaE\trE(h)$, then $\Scal^\lin(\eta)=\mathcal O_\infty(r^{-5})$. Plugging this into $\Scal(g)=0$ gives $\Scal^\lin(h^A)=\mathcal O_\infty(r^{-5})$. On the other hand, (denote $\theta_i=\tfrac{x_i}{r}$)
		\begin{align*}
			\trE(h^A)
			=
			U_{k\ell}\,\frac{x_kx_\ell}{r^4}
			=
			r^{-2}U_{k\ell}\theta_k\theta_\ell 
			=
			r^{-2}\left(\frac14\tr(U)+(U_{k\ell})_0\,\theta_k\theta_\ell\right),
		\end{align*}
		thus
		\begin{align*}
			\Scal^\lin(h^A)
			=
			-\frac{1}{2}\DeltaE\left[r^{-2}\left(\frac14\tr(U)+(U_{k\ell})_0\,\theta_k\theta_\ell\right)\right] 
			=
			4\,r^{-4}(U_{k\ell})_0\,\theta_k\theta_\ell,
		\end{align*}
		Since this is $\mathcal O_\infty(r^{-5})$, we get
		\begin{align*}
			(U_{k\ell})_0=0.
		\end{align*}
		So to prove \eqref{eq:vanish-result}, it suffices to show that $\bigl(V_{(i\ell)}\bigr)_0$ also vanishes. 
		
		Subtracting from \eqref{eq:bianchi-trace-identity} the same identity with $i$ and
		$\ell$ interchanged, and using the pair symmetry of $A$, we obtain $V_{i\ell}=V_{\ell i}$. Therefore, the goal of the proof has been reduced from $\bigl(V_{(i\ell)}-U_{i\ell}\bigr)_0=0$ to $\bigl(V_{i\ell}\bigr)_0=0$.
		
		Now, let us denote the trace-less symmetric 2-tensor $\bigl(V_{(i\ell)}-U_{i\ell}\bigr)_0$ by $\mathcal S$. Then the symmetric 2-tensors $U$ and $V$ can be decomposed respectively as
		\begin{align*}
			V_{i\ell}=\frac{\tr(V)}{4}\delta_{i\ell}+\mathcal S_{i\ell};
			\quad
			U_{k\ell}=\frac{\tr(U)}{4}\delta_{k\ell}.
		\end{align*}
		Substituting these decompositions into
		\eqref{eq:bianchi-tensor-form}, we obtain
		\begin{align*}
			A_{i(jk\ell)}
			&=
			\frac16\Bigl(
			\delta_{jk}\mathcal S_{i\ell}
			+\delta_{j\ell}\mathcal S_{ik}
			+\delta_{k\ell}\mathcal S_{ij}
			\Bigr)
			\\
			&\quad
			+
			\frac{2\tr(V)+\tr(U)}{48}\Bigl(
			\delta_{jk}\delta_{i\ell}
			+\delta_{j\ell}\delta_{ik}
			+\delta_{k\ell}\delta_{ij}
			\Bigr)
			\\
			&=
			\frac16\Bigl(
			\delta_{jk}\mathcal S_{i\ell}
			+\delta_{j\ell}\mathcal S_{ik}
			+\delta_{k\ell}\mathcal S_{ij}
			\Bigr).
		\end{align*}
		Since $\tr(S)=0$, setting $i=j=p$ in the above identity yields
		\begin{align*}
			A_{p(pk\ell)}
			=
			\frac13\,\mathcal S_{k\ell}.
		\end{align*}
		On the other hand, expanding the symmetrization and using the pair symmetries of $A$ and $V$, we have
		\begin{align*}
			A_{p(pk\ell)}
			=
			\frac16\Bigl(
			A_{ppk\ell}
			+
			A_{pp\ell k}
			+
			A_{pkp\ell}
			+
			A_{pk\ell p}
			+
			A_{p\ell pk}
			+
			A_{p\ell kp}
			\Bigr)
			=
			\frac13\bigl(U_{k\ell}+2V_{k\ell}\bigr).
		\end{align*}
		Therefore
		\begin{align*}
			U_{k\ell}+2V_{k\ell}=\mathcal S_{k\ell}.
		\end{align*}
		Taking trace-free parts and using $U_0=0$ and $V_0=\mathcal S$, we obtain \eqref{eq:vanish-result}.
	\end{proof}
	
	\medskip
	
	So far we have identified the main obstruction to the vanishing of the partial traces of the rough coefficient $A$, namely the potential failure of $\tr(V)+2\tr(U)=0$. This leads naturally to the next step: we split off from $A$ the part that gives rise to the failure.
	
	\subsection{Canonical splitting of the rough leading coefficient}
	\label{subsec:canonical-splitting}
	
	Continuing from the end of the previous subsection, since directly rendering the rough coefficient $A$ in \eqref{eq:h-leading-cartesian} partial trace-free via the coordinate transformation $x_i=y_i+B_{ij}\tfrac{y_j}{|y|^2}$ encounters an obstruction, which is quantified by $\tr(V)+2\tr(U)$, we relax the requirement. Instead of requiring both partial traces to vanish, we seek a coordinate transformation such that the transformed $A$ satisfies
	\begin{align*}
		A_{ppk\ell}=0,
		\qquad
		A_{ijj\ell}=A_{\ell jji}.
	\end{align*}
	That is, we kill the first partial trace and only require the second one to be symmetric. This weaker system holds promise of being solvable. The point of this adjustment is that it allows us to separate from $A$ the part responsible for the failure of the partial trace-free condition, while the remaining part can then be arranged to satisfy both partial trace-free identities.
	
	Specifically, if we seek a constant matrix $B$ within the orbit $A^\sharp=A+\Psi_4(B)$, where $\Psi_4$ is defined by \eqref{eq:Psi4-def} and $A$ satisfies $A_{ijk\ell}=A_{jik\ell}=A_{ij\ell k}$, such that $A^\sharp_{ppk\ell}=0$ and $A^\sharp_{ijj\ell}=A^\sharp_{\ell jji}$, it actually suffices to choose $B$ as follows:
	\begin{align}\label{eq:solution-B}
		\tr(B)=-\frac14\tr(U),
		\quad
		B_{(k\ell)}=\frac14U_{k\ell}-\frac18\tr(U)\delta_{k\ell},
		\quad
		B_{[i\ell]}=\frac16\,V_{[i\ell]}.
	\end{align}
	With this choice, set
	\begin{align*}
		\lambda(A):=\frac14\bigl(\tr(V)+2\,\tr(U)\bigr)
	\end{align*}
	and define
	\begin{align}\label{eq:Xi-def}
		\Xi(\mu)_{ijk\ell}
		:=
		-\frac{\mu}{9}\,\delta_{ij}\delta_{k\ell}
		+
		\frac{2\mu}{9}
		\bigl(\delta_{ik}\delta_{j\ell}+\delta_{i\ell}\delta_{jk}\bigr).
	\end{align}
	We then separate the scalar part from $A^\sharp$ by setting
	\begin{align}\label{eq:separate-first}
		W(A):=A^\sharp-\Xi\bigl(\lambda(A)\bigr).
	\end{align}
	The following proposition shows that, after the scalar obstruction $\lambda(A)$ has been tracked and removed, the remaining part has the partial trace-free property
	required of a Weyl-type term.
	
	\begin{proposition}\label{prop:W-properties}
		With $B$, $A^\sharp$ and $W(A)$ defined as above, one has
		\begin{align*}
			W(A)_{ppk\ell}=0,
			\qquad
			W(A)_{ijj\ell}=0.
		\end{align*}
	\end{proposition}
	
	\begin{proof}
		The map $\Xi$ possesses the following partial trace property:
		\begin{align*}
			\Xi(\mu)_{ppk\ell}=0,
			\qquad
			\Xi(\mu)_{ijj\ell}=\mu\,\delta_{i\ell},
		\end{align*}
		hence $W(A)_{ppk\ell}=0$ holds obviously. As for the second property $W(A)_{ijj\ell}=0$, we compute
		\begin{align*}
			W(A)_{ijj\ell}
			&=A_{ijj\ell}+\Psi_4(B)_{ijj\ell}-\Xi\bigl(\lambda(A)\bigr)_{ijj\ell}
			\\[3pt]
			&=V_{i\ell}-4B_{(i\ell)}-6B_{[i\ell]}-\tr(B)\delta_{i\ell}-\lambda(A)\delta_{i\ell}
			\\
			&=V_{i\ell}-\frac{\tr(V)}{4}\delta_{i\ell}=(V_{i\ell})_0=0,
		\end{align*}
		where we used Proposition~\ref{prop:residual-charge-vanishes} in the last step. Proof complete.
	\end{proof}
	
	\medskip
	
	At this stage, after separating out the obstruction term $\lambda(A)=\tfrac14(\tr(V)+2\,\tr(U))$ as in \eqref{eq:separate-first}, the remainder tensor $W(A)$ already satisfies the two partial trace-free conditions. In other words, recalling the goal we set at the beginning of this section—to trace the three properties required of the Weyl-type tensor in Proposition~\ref{prop:Wtilde-eq}—we have so far successfully traced the first two properties from the leading coefficient $A$ of the rough expansion: pair symmetry and partial trace-free. Therefore, in what follows, we proceed to trace the remaining third property from $W(A)$: the first Bianchi identity.
	
	\medskip
	
	For convenience, we introduce the notation $T_{\cyc}$ to denote the cyclic sum $T_{ijk\ell} + T_{ik\ell j} + T_{i\ell jk}$ associated with a tensor $T$. We now compute $W(A)_{\cyc}$ using the decomposition~\eqref{eq:separate-first}, i.e. 
	\begin{align}\label{eq:decompose-WA-cyc}
		W(A)_{\cyc}=A^\sharp_{\cyc}-\Xi(\lambda(A))_{\cyc}.
	\end{align}
	We need to compute terms in RHS of \eqref{eq:decompose-WA-cyc} separately.
	
	It's worthy to notice that, Proposition~\ref{prop:residual-charge-vanishes} implies both $U$ and $V$ are pure-trace, i.e.
	\begin{align*}
		U_{k\ell}=\frac{\tr(U)}{4}\delta_{k\ell};
		\qquad
		V_{i\ell}=\frac{\tr(v)}{4}\delta_{i\ell}.
	\end{align*}
	
	For $\Xi(\lambda(A))_{\cyc}$, using Definition~\ref{eq:Xi-def},
	\begin{align}\label{eq:compute Xi(lambda)_cyc}
		\Xi(\lambda(A))_{\cyc}
		=\left( \frac{\tr(V)}{12}+\frac{\tr(U)}{6}\right)
		\left(\delta_{ik}\delta_{\ell j}
		+ \delta_{i\ell}\delta_{jk}
		+\delta_{ij}\delta_{k\ell}\right).
	\end{align}
	
	For $A^\sharp$, notice that $A^\sharp_\cyc=A_\cyc+\Psi_4(B)_\cyc$, therefore, we need to compute $A_\cyc$ and $\Psi_4(B)_\cyc$. 
	
	First, one may exploit the constraint imposed by the Bianchi gauge—namely, the algebraic identity \eqref{eq:bianchi-tensor-form}—to rewrite $A_{\cyc}$; that is,
	\begin{align}\label{eq:compute A_cyc}
		\begin{split}
			A_\cyc=\left( \frac{\tr(V)}{8}+\frac{\tr(U)}{16}\right) \left( \delta_{jk}\delta_{i\ell}
			+\delta_{j\ell}\delta_{ik}+\delta_{k\ell}\delta_{ij}\right).
		\end{split}
	\end{align}
	
	Similarly, the Definition~\ref{eq:Psi4-def} reads
	\begin{align*}
		\Psi_4(B)_\cyc
		=
		-2\Bigl(\delta_{k\ell}B_{[ij]}
		+\delta_{j\ell}B_{[ik]}+\delta_{jk}B_{[i\ell]}\Bigr)-2\Bigl(
		\delta_{i\ell}B_{(jk)}+\delta_{ik}B_{(j\ell)}+\delta_{ij}B_{(k\ell)}\Bigr)
	\end{align*}
	then using the ``solution'' $B$ (defined in \eqref{eq:solution-B}), $V_{[ij]}=0$ and $U_{ij}=\tfrac{\tr(U)}{4}\delta_{ij}$, above expression can be reformulated as
	\begin{align}\label{eq:compute-Psi4B-cyc}
		\Psi_4(B)_\cyc
		=\left( -\frac{\tr(U)}{8}-\frac{\tr(V)}{2}\right) \Bigl(\delta_{i\ell}\delta_{jk}
		+\delta_{ik}\delta_{j\ell}+\delta_{ij}\delta_{k\ell}\Bigr).
	\end{align}
	
	Combining \eqref{eq:compute A_cyc} with \eqref{eq:compute-Psi4B-cyc} yields
	\begin{align}\label{eq:compute-Asharp-cyc}
		A^\sharp_\cyc
		=
		\left(-\frac{\tr(U)}{16}-\frac{3\tr(V)}{8} \right) \Bigl(\delta_{i\ell}\delta_{jk}
		+\delta_{ik}\delta_{j\ell}+\delta_{ij}\delta_{k\ell}\Bigr)
	\end{align}
	
	Plugging \eqref{eq:compute-Asharp-cyc}, \eqref{eq:compute Xi(lambda)_cyc} into \eqref{eq:decompose-WA-cyc} gives
	\begin{align}\label{eq:compute-Wcyc}
		W(A)_{cyc}
		=
		\left(-\frac{11\tr(U)}{48}-\frac{22\tr(V)}{48} \right) \Bigl(\delta_{i\ell}\delta_{jk}
		+\delta_{ik}\delta_{j\ell}+\delta_{ij}\delta_{k\ell}\Bigr)=0,
	\end{align}
	here we used $\tr(U)+2\tr(V)=0$.
	
	Therefore, in view of \eqref{eq:compute-Wcyc}, plugging the properties $W(A)_{ppk\ell}=0,
	W(A)_{ijj\ell}=0$ previously obtained, we deduce that $W(A)$ is a Weyl-type tensor, i.e. $W(A)\in \widetilde W$, thanks to Proposition~\ref{prop:Wtilde-eq}. So far, we have accomplished the goal set forth at the beginning of this section: to track the Weyl-type tensor from the leading coefficient $A$ of the rough expansion \eqref{eq:h-leading-cartesian}, based on the three equivalent properties established in Proposition~\ref{prop:Wtilde-eq}. In other words, after performing the coordinate transformation $x_i=y_i+B_{ij}\frac{y_j}{|y|^2}$ based on \eqref{eq:solution-B}, the leading coefficient $A^\sharp$ admits the splitting:
	\begin{align}\label{eq:Asharp-splitting}
		A^\sharp=s\left( W_\infty\right) +\Xi\bigl(\lambda(A)\bigr),
	\end{align}
	where $W_\infty$ is a Weyl tensor, the symmetrization map $s$ defined in \eqref{eq:def-s-map}, and $\lambda(A)=\tfrac14(\tr(V)+2\tr(U))$ to be determined. For convenience, we shall refer to these two terms as the Weyl-type term and the scalar term, respectively.

	\subsection{The scalar $\lambda(A)$ and the ALE ADM mass}
	
	At the end of the previous subsection we obtained the splitting \eqref{eq:Asharp-splitting}, in which the geometric interpretation of $s(\Weyl_\infty(A))$ is clear. We still needs to identify the scalar quantity $\lambda(A)$.
	
	For ALE metrics, the ADM mass is the basic scalar invariant defined from the asymptotic behavior of the metric at infinity. It is therefore natural to compare the scalar part $\lambda(A)$ in \eqref{eq:Asharp-splitting} with the ADM mass. We first recall the definition of the ALE ADM mass.
	
	On the universal cover of an ALE end modeled on $\R^4/\Gamma$, let $S_r\subset\R^4$ denote the Euclidean sphere of radius $r$, with outward unit normal $\nu$ and Euclidean surface measure $d\mu$. In dimension $4$, write $\omega_3:=|S^3|=2\pi^2$, and define the ADM mass by
	\begin{align}\label{eq:mass-def-prep}
		m_{\ADM}(g):=\frac{1}{6\,\omega_3\,|\Gamma|}\lim_{r\to\infty}\int_{S_r}
		\bigl(\partial_j g_{ij}-\partial_i g_{jj}\bigr)\,\nu_i\,d\mu,
	\end{align}
	whenever the limit exists. Under the asymptotic model produced above, the metric already has a
	Cartesian $r^{-2}$ leading term, so the integrand in \eqref{eq:mass-def-prep} is of order $O(r^{-3})$.
	
	Actually, in our ALE geometry, given the asymptotic expansion of the metric $g_{ij}(x)=\delta_{ij}+h_{ij}(x)$, where
	\begin{align*}
		h_{ij}(x)=A^\sharp_{ijk\ell}\tfrac{x_kx_\ell}{r^4}+\mathcal O_\infty(r^{-3}),
	\end{align*}
	a straightforward fact is that, the ADM mass is determined by the $r^{-2}$ term alone:
	\begin{align}\label{eq:mass-reduce-to-h2}
		m_{\ADM}(g)
		=
		\frac{1}{6\,\omega_3\,|\Gamma|}
		\lim_{r\to\infty}\int_{S_r}
		\Bigl(
		\partial_j\Bigl(A^\sharp_{ijk\ell}\frac{x_kx_\ell}{r^4}\Bigr)
		-
		\partial_i\Bigl(A^\sharp_{jjk\ell}\frac{x_kx_\ell}{r^4}\Bigr)
		\Bigr)\nu_i\,dS.
	\end{align}
	This reduces the problem to the homogeneous coefficient $A^\sharp$ itself.
	Motivated by \eqref{eq:mass-reduce-to-h2}, for any constant tensor
	$A^\sharp_{ijk\ell}$ satisfying pair-symmetry, i.e. $A^\sharp_{ijk\ell}=A^\sharp_{jik\ell}=A^\sharp_{ij\ell k}$, we denote $A^\sharp_{ijk\ell}\tfrac{x_kx_\ell}{r^4}$ by $h^{(2)}[A^\sharp]_{ij}(x)$, and introduce the linear functional
	\begin{align}\label{eq:mass-functional}
		\mathfrak m(A^\sharp)
		:=
		\frac{1}{6\,\omega_3\,|\Gamma|}
		\lim_{r\to\infty}\int_{S_r}
		\bigl(
		\partial_j h^{(2)}[A^\sharp]_{ij}
		-
		\partial_i h^{(2)}[A^\sharp]_{jj}
		\bigr)\nu_i\,dS.
	\end{align}
	For the tensor $A^\sharp$ arising from the ALE expansion of $g$,
	\eqref{eq:mass-reduce-to-h2} is exactly the statement that
	\begin{align}\label{eq:mass-as-functional}
		m_{\ADM}(g)=\mathfrak m(A^\sharp).
	\end{align}
	Thus the remaining task is to compare the two scalar quantities $\mathfrak m(A^\sharp)$ and $\lambda(A)$. The key point is that $\mathfrak m$ is linear and invariant under the natural
	$\mathrm O(4)$-action on the space of pair-symmetric tensors. Therefore
	$\mathfrak m$ can only depend on the $\mathrm O(4)$-invariant part of
	$A^\sharp$. We use the standard classification of isotropic fourth-order tensors
	with minor symmetries; see \cite[(5.1)--(5.3)]{Moakher}.
	
	\begin{lemma}\label{lem:isotropic-pairsym}
		Let $V=\R^{4}$ with the standard Euclidean metric $\gE$, and let $\mathcal T$
		be the vector space of constant $(0,4)$-tensors $A^\sharp_{ijk\ell}$ satisfying $A^\sharp_{ijk\ell}=A^\sharp_{jik\ell}=A^\sharp_{ij\ell k}$. Equip $\mathcal T$ with the natural $\mathrm O(4)$-action
		\begin{align*}
			(A\cdot A^\sharp)_{ijk\ell}:=A_{ai}A_{bj}A_{ck}A_{d\ell}\,A^\sharp_{abcd}.
		\end{align*}
		Then the fixed subspace $\mathcal T^{\mathrm O(4)}$ is two-dimensional, and
		\begin{align*}
			A^\sharp\in \mathcal T^{\mathrm O(4)}
			\quad\Longleftrightarrow\quad
			A^\sharp_{ijk\ell}
			=
			\alpha\,\delta_{ij}\delta_{k\ell}
			+
			\beta\,(\delta_{ik}\delta_{j\ell}+\delta_{i\ell}\delta_{jk})
		\end{align*}
		for some $\alpha,\beta\in\R$. In particular,
		$\delta_{ij}\delta_{k\ell}$ and $\delta_{ik}\delta_{j\ell}+\delta_{i\ell}\delta_{jk}$ form a basis of
		$\mathcal T^{\mathrm O(4)}$.
	\end{lemma}
	
	\medskip
	
	The previous lemma describes all $\mathrm O(4)$-invariant tensors in the ambient
	space $\mathcal T$. In our situation, however, the tensors arising from the ALE
	expansion satisfy one more algebraic condition, namely $A^\sharp_{ppk\ell}=0$. We therefore restrict to the subspace
	\begin{align*}
		\mathcal V
		:=
		\bigl\{
		A^\sharp\in\mathcal T:\ A^\sharp_{ppk\ell}=0
		\bigr\}.
	\end{align*}
	
	Our goal is to compute the scalar quantity $\mathfrak m(A^\sharp)$, as claimed before. Since $\mathfrak m$ is linear and $\mathrm O(4)$-invariant, only the
	invariant component of $A^\sharp$ can contribute to its value. We isolate this
	component by the standard compact-group averaging procedure \cite{FultonHarris}.
	Define
	\begin{align*}
		P:\mathcal V\to\mathcal V,
		\qquad
		P(A^\sharp):=\int_{\mathrm O(4)}A\cdot A^\sharp\,dA,
	\end{align*}
	where $dA$ denotes the normalized Haar measure on $\mathrm O(4)$.
	
	The space $\mathcal V$ is preserved by the natural $\mathrm O(4)$-action, so $P$ is
	well-defined. By the left-invariance of the Haar measure, $P(A^\sharp)$ is
	$\mathrm O(4)$-invariant, and hence $P(A^\sharp)\in\mathcal V^{\mathrm O(4)}$.
	Moreover, since $\mathfrak m$ is $\mathrm O(4)$-invariant and linear,
	\begin{align*}
		\mathfrak m(A^\sharp)
		&=
		\int_{\mathrm O(4)}\mathfrak m(A\cdot A^\sharp)\,dA \\
		&=
		\mathfrak m\left(\int_{\mathrm O(4)}A\cdot A^\sharp\,dA\right)
		=
		\mathfrak m\bigl(P(A^\sharp)\bigr),
	\end{align*}
	Thus, in computing $\mathfrak m(A^\sharp)$, we may replace $A^\sharp$ by its $\mathrm O(4)$-average.
	
	It remains to understand $\mathcal V^{\mathrm O(4)}$. Since
	\begin{align*} 
		\Xi(1)_{ijk\ell} = -\frac19\,\delta_{ij}\delta_{k\ell} + \frac29\,(\delta_{ik}\delta_{j\ell}+\delta_{i\ell}\delta_{jk}).
	\end{align*}
	satisfies $\Xi(1)_{ppk\ell}=0$, we have $\Xi(1)\in\mathcal V^{\mathrm O(4)}$. Conversely, if $A^\sharp\in\mathcal V^{\mathrm O(4)}$, then by
	Lemma~\ref{lem:isotropic-pairsym} we may write
	\begin{align*}
		A^\sharp_{ijk\ell}
		=
		\alpha\,\delta_{ij}\delta_{k\ell}
		+
		\beta\,(\delta_{ik}\delta_{j\ell}+\delta_{i\ell}\delta_{jk})
	\end{align*}
	for some $\alpha,\beta\in\R$. Since $A^\sharp\in\mathcal V$, contracting in the
	first pair gives $\beta=-2\alpha$. It follows that $\mathcal V^{\mathrm O(4)}$ is
	one-dimensional, and
	\begin{align*}
		\mathcal V^{\mathrm O(4)}=\R\,\Xi(1).
	\end{align*}
	
	Consequently, for every $A^\sharp\in\mathcal V$, there exists a unique scalar
	$c(A^\sharp)\in\R$ such that
	\begin{align*}
		P(A^\sharp)=c(A^\sharp)\,\Xi(1).
	\end{align*}
	Therefore
	\begin{align}\label{eq:compute-map-m}
		\mathfrak m(A^\sharp)
		=
		\mathfrak m\bigl(P(A^\sharp)\bigr)
		=
		c(A^\sharp)\,\mathfrak m(\Xi(1)).
	\end{align}
	To identify $\mathfrak m(A^\sharp)$, or the ALE ADM mass $m_\ADM(g)$ in view of \eqref{eq:mass-as-functional}, it remains to compute the
	coefficient $c(A^\sharp)$ and the number $\mathfrak m(\Xi(1))$.
	
	\begin{itemize}
		\item \textit{the computation of $c(A^\sharp)$:} The full scalar contraction
		$A^\sharp_{ijji}$ is also $\mathrm O(4)$-invariant:
		\begin{align*}
			(A\cdot A^\sharp)_{ijji}
			=
			A_{ai}A_{bj}A_{cj}A_{di}\,A^\sharp_{abcd} 
			=
			\delta_{bc}\delta_{ad}\,A^\sharp_{abcd}
			=
			A^\sharp_{abba}.
		\end{align*}
		Hence
		\begin{align*}
			A^\sharp_{ijji}
			=
			\bigl(P(A^\sharp)\bigr)_{ijji}
			=
			c(A^\sharp)\,\Xi(1)_{ijji}
			=
			4\,c(A^\sharp).
		\end{align*}
		Therefore, since 
		\begin{align*}
			\lambda(A)=\frac14\Bigl(\tr(V)+2\tr(U)\Bigr)=\frac14 A^\sharp_{ijji},
		\end{align*}
		then
		\begin{align*}
			c(A^\sharp)=\lambda(A).
		\end{align*}
		\item the computation of $\mathfrak m(\Xi(1))$
		
		From the formula for $\Xi(1)$ we get
		\begin{align*}
			\partial_j h^{(2)}[\Xi(1)]_{ij}
			=
			\frac23\,x_ir^{-4},
			\qquad
			\partial_i h^{(2)}[\Xi(1)]_{jj}=0.
		\end{align*}
		Substituting into the definition \eqref{eq:mass-functional} for $\mathfrak m$, we find
		\begin{align*}
			\mathfrak m(\Xi(1))
			=
			\frac{1}{6\,\omega_3\,|\Gamma|}\cdot\frac23\,\omega_3
			=
			\frac{1}{9\,|\Gamma|}.
		\end{align*}
	\end{itemize}
	Combining these two computations into \eqref{eq:compute-map-m} yields
	\begin{align*}
		m_\ADM=\mathfrak m(A^\sharp)=\frac{1}{9\,|\Gamma|}\,\lambda(A).
	\end{align*}
	
	We now summarize the current computation in the following proposition, which clarifies the connection between the scalar term $\Xi\bigl(\lambda(A)\bigr)$ in the splitting \eqref{eq:Asharp-splitting} and the ALE ADM mass.
	
	\begin{proposition}\label{prop:lambda-mass}
		In the splitting \eqref{eq:Asharp-splitting} we derived before, the scalar term $\Xi(\lambda)$ is actually determined by the ALE ADM mass of the end. More precisely, \eqref{eq:Asharp-splitting} can be further rewritten as
		\begin{align}\label{eq:Asharp-splitting-version-2}
			A^\sharp
			=
			\Xi\bigl(9|\Gamma|m_\ADM\bigr)
			+
			s\bigl(\Weyl_\infty(A)\bigr).
		\end{align}
	\end{proposition}
	
	\medskip
	
	Ultimately, the true identity of the unknown term $\Xi(\lambda(A))$ in the splitting of the leading coefficient $A^\sharp$ has been identified. This suggests a potential connection between the ALE geometric background under consideration and General Relativity.
	
	\subsection{A Schwarzschild background and a relative repair}
	\label{subsec:Schw-anchored-repair}
	
	After performing the coordinate elimination, the homogeneous $r^{-2}$
	coefficient in the rough expansion \eqref{eq:h-leading-cartesian} has been reduced to the form $A^\sharp$. However, the Euclidean Bianchi gauge obtained earlier is in
	general no longer preserved. It is therefore natural to ask whether one can repair
	the gauge once more, as was done in the Ricci-flat case in \cite{WY},
	
	Unfortunately, a direct Euclidean Bianchi correction is not suitable here. Indeed, by
	\eqref{eq:Asharp-splitting-version-2}, the obstruction is exactly the scalar term $\Xi(\lambda)$, where
	$\lambda:=\lambda(A)=9|\Gamma|m_{\ADM}(g)$. Trying to remove it directly produces a correction vector field of order $r^{-1}$,
	whose Lie derivative changes the homogeneous $r^{-2}$ term. So this repair does not preserve $A^\sharp$. We therefore proceed differently. We will absorb the scalar part of
	$A^\sharp$ into a reference Schwarzschild metric, and then impose a relative gauge with respect to
	this background.
	
	\medskip
	
	Now we introduce the reference Schwarzschild metric. Set $m_{\Sch}:=\frac{\lambda}{9}=|\Gamma|m_{\ADM}(g)$, consider the isotropic Schwarzschild metric on the cover $\R^4\setminus B_R$,
	\begin{align*}
		U(x):=1+\frac{\lambda}{18\,r^2}
		=
		1+\frac{m_{\Sch}}{2\,r^2},
		\qquad
		g^{\Sch}:=U^2\gE.
	\end{align*}
	Since $\DeltaE(r^{-2})=0$ on $\R^4\setminus\{0\}$, the metric $g^{\Sch}$ is
	scalar-flat on every exterior region.
	
	To match the form $h^{(2)}[\Xi(\lambda)]$, we apply the radial
	change of variables $\Phi^\Sch(x):=x-|\Gamma|m_{\ADM}(g)\frac{x}{r^2}$, and define $g^{\Sch,\sharp}:=(\Phi^\Sch)^*g^{\Sch}$. Then
	\begin{align}\label{eq:gSch-sharp-expansion}
		g^{\Sch,\sharp}
		=
		\gE+h^{(2)}[\Xi(\lambda)]+\mathcal O_\infty(r^{-4}),
	\end{align}
	where, for any constant pair-symmetric $(0,4)$-tensor $A$, $h^{(2)}[A]_{ij}(x):=A_{ijk\ell}\tfrac{x_kx_\ell}{r^4}$.
	
	From the construction, $g^{\Sch,\sharp}$ encodes the mass and acts as the reference metric naturally associated with the end. Consequently, it provides the background for the final relative gauge.
	
	With the reference metric established, we define the relative gauge considered in this paper as follows:
	
	\begin{definition}\label{def:Schw-relative-gauge}
		For two metrics $g,\bar g$ on $\Omega_R$, define
		\begin{align*}
			\mathscr V(g,\bar g)^k
			:=
			g^{ij}\Bigl(\Gamma(g)^k_{ij}-\Gamma(\bar g)^k_{ij}\Bigr).
		\end{align*}
		We say that $g$ is in relative gauge with respect to $g^{\Sch,\sharp}$ if
		\begin{align*}
			\mathscr V\left( g,g^{\Sch,\sharp}\right) =0
			\qquad\text{on }\Omega_R.
		\end{align*}
	\end{definition}
	
	\medskip
	
	Just as in the construction of the Euclidean Bianchi gauge, we reformulate the problem of finding coordinates in the relative gauge as an equation for a vector field. Let $\widetilde{g}$ be the metric resulting from the coordinate transformation, where the coefficient of the $-2$-order term is given by $A^\sharp$. Let $\varphi(x)=x+w(x)$ with $w\to0$ at infinity, and set
	$g:=\varphi^*\widetilde g$. To require that $g$ be in relative gauge with respect to $g^{\Sch,\sharp}$ is exactly to require
	\begin{align*}
		\mathscr V\left( \varphi^*\widetilde g,g^{\Sch,\sharp}\right) =0
		\qquad\text{on }\Omega_R.
	\end{align*}
	We therefore define
	\begin{align*}
		\mathscr F_{\Sch}(w):=\mathscr V\left( \varphi^*\widetilde g,g^{\Sch,\sharp}\right),
	\end{align*}
	and seek $w$ as a zero of the mapping $\mathscr F_{\Sch}$.
	
	The strategy remains the same: we linearize $\mathscr F_{\Sch}$ at $w=0$, isolate $\Delta_E$ as the principal part, and then solve the resulting equation by the same fixed-point argument as in the Euclidean Bianchi gauge-fixing.
	
	For arbitrary fixed $\alpha\in(4,\infty)\setminus\N$ and $\sigma\in(0,1)$, there exist
	$R$ sufficiently large and $\varepsilon>0$ sufficiently small such that, for every
	$w\in \mathcal X_{\alpha+1,\,1+\sigma}(\Omega_R)$ with
	$\|w\|_{\mathcal X_{\alpha+1,\,1+\sigma}(\Omega_R)}<\varepsilon$, one has
	\begin{align*}
		\mathscr F_{\Sch}(w)
		=
		\mathscr V(\widetilde g,\bar g_{\lambda,\mathcal S})
		+
		\DeltaE w
		+
		\mathscr N_{\Sch}(w),
	\end{align*}
	where $\mathscr N_{\Sch}$ satisfies the bounds
	\begin{align*}
		\|\mathscr N_{\Sch}(w)\|_{\mathcal X_{\alpha-1,\,3+\sigma}(\Omega_R)}
		&\le
		C\Bigl(
		R^{-2}\,\|w\|_{\mathcal X_{\alpha+1,\,1+\sigma}(\Omega_R)}
		+
		R^{-(2+\sigma)}\,\|w\|_{\mathcal X_{\alpha+1,\,1+\sigma}(\Omega_R)}^2
		\Bigr),
	\end{align*}
	and
	\begin{align*}
		\|\mathscr N_{\Sch}(w_1)-\mathscr N_{\Sch}(w_2)\|_{\mathcal X_{\alpha-1,\,3+\sigma}(\Omega_R)}
		\le
		C\Bigl(
		R^{-2}
		+
		R^{-(2+\sigma)}
		(\|w_1\|+\|w_2\|)
		\Bigr)\,
		\|w_1-w_2\|_{\mathcal X_{\alpha+1,\,1+\sigma}(\Omega_R)}.
	\end{align*}
	The desired gauge repair then follows from the same fixed-point argument as before. We therefore omit the repetition and state the resulting proposition.
	
	\begin{proposition}\label{prop:Schw-anchored-repair}
		There exists $\sigma\in(0,1)$, $\alpha\in(4,\infty)\setminus\N$, $R$ sufficiently large and $w\in\mathcal X_{\alpha+1,1+\sigma}(\Omega_R)$, such that $\varphi(x)=x+w(x)$ is a $C^1$-diffeomorphism of $\Omega_R$ onto its image,
		$\varphi(\Omega_R)\subset\Omega_{R/2}$, and
		\begin{align*}
			\mathscr V\left(\varphi^*\widetilde g,g^{\Sch,\sharp}\right)=0
			\qquad\text{on }\Omega_R.
		\end{align*}
		Moreover,
		\begin{align*}
			\|w\|_{\mathcal X_{\alpha+1,\,1+\sigma}(\Omega_R)}
			\le
			C\,\|\mathscr V(\widetilde g,g^{\Sch,\sharp})\|_{\mathcal X_{\alpha-1,\,3+\sigma}(\Omega_R)}.
		\end{align*}
		If $g:=\varphi^*\widetilde g$, then
		\begin{align*}
			g_{ij}(x)
			=
			\delta_{ij}
			+
			A^\sharp_{ijk\ell}\frac{x_kx_\ell}{r^4}
			+
			\mathcal O_\infty(r^{-2-\sigma}),
		\end{align*}
		so the homogeneous $r^{-2}$ coefficient is unchanged.
	\end{proposition}
	
	\begin{remark}
		In what follows, we will refer to a coordinate system as a \textit{preferred coordinate system} if it simultaneously satisfies the condition that the coefficient of the order $-2$ term in the asymptotic expansion of the metric at infinity is $A^\sharp$, and the relative gauge considered in Proposition~\ref{prop:Schw-anchored-repair}.
	\end{remark}
	
	\medskip
	
	Proposition~\ref{prop:Schw-anchored-repair} shows that the relative gauge can be imposed without changing the homogeneous $r^{-2}$ coefficient $A^\sharp$. So far, we have completed the proof of the first main theorem, Theorem~\ref{thm:main1}.
	
	\subsection{Rigidity and Euclidean motion at infinity}
	
	Up to this point, we have thoroughly investigated the metric at infinity, with a specific focus on the coefficient of the $-2$-order term. A basic issue remains: is the curvature information encoded in this term an intrinsic feature
	of the end, or is it still affected by the residual choice of coordinates? This is the main uniqueness problem
	in the asymptotic analysis. Unless it is resolved, the coefficient $A^\sharp$ and the quantities extracted from
	it cannot yet be regarded as geometric invariants of the end.
	
	In this subsection we show that the remaining freedom is exactly the expected one. Any two preferred
	coordinate systems on the same end differ, to leading order, by a Euclidean motion at infinity. This rigidity
	gives a precise transformation law for $A^\sharp$, and shows that the data
	$(\lambda,\Weyl_\infty)$ obtained from its canonical splitting are intrinsic, modulo the natural
	$O(4)$-action. Thus the $r^{-2}$ term is not merely a convenient normal form: it provides a canonical way
	to identify the curvature information at infinity.
	
	We begin with the rigidity statement for the transition map.
	
	\begin{theorem}\label{thm:main-rigid-bootstrapped}
		Assume $(M^4,g)$ admits two preferred coordinate systems $x$ and $y$ on the same end
		(after lifting to the universal cover if necessary), both defined on exterior domains $\Omega_R=\R^4\setminus B_R$ and $\Omega_{R'}=\R^4\setminus B_{R'}$. Let $g^{(x)}$ and $g^{(y)}$ be the coordinate expressions of $g$,
		and let
		\begin{align*}
			F:=y\circ x^{-1}:\Omega_R\to\Omega_{R'}.
		\end{align*}
		Assume that for some $\sigma\in(0,1)$ one has
		\begin{align}
			\begin{split}
				g^{(x)}_{ij}(x)
				&=
				\delta_{ij}
				+
				A^{\sharp,(x)}_{ijk\ell}\frac{x_kx_\ell}{|x|^4}
				+
				\mathcal O_\infty(|x|^{-2-\sigma}),
				\\
				g^{(y)}_{ab}(y)
				&=
				\delta_{ab}
				+
				A^{\sharp,(y)}_{abpq}\frac{y_py_q}{|y|^4}
				+
				\mathcal O_\infty(|y|^{-2-\sigma}),
			\end{split}
			\label{eq:gx-gy-decay-assumption-SR}
		\end{align}
		where $A^{\sharp,(x)}$ and $A^{\sharp,(y)}$ are constant pair-symmetric tensors. Then there exist $A\in O(4)$ and $b\in\R^4$ such that
		\begin{align}
			F(x)=Ax+b+\mathcal E(x),
			\label{eq:F-affine-bootstrapped-SR}
		\end{align}
		with
		\begin{align}
			\mathcal E(x)=\mathcal O_\infty(|x|^{-2}\log|x|),
			\qquad
			\partial^{(k)}\mathcal E(x)=\mathcal O_\infty(|x|^{-2-k}\log|x|)\quad (k\ge1).
			\label{eq:E-bootstrapped-decay-SR}
		\end{align}
		If both preferred charts are $\Gamma$-compatible, then $A$ normalizes $\Gamma$.
	\end{theorem}
	
	The proof uses two standard analytic inputs. The first is a standard exterior Poisson estimate in dimension four. We record only the form needed below and omit the proof. Briefly, one multiplies $u$ by a cutoff to obtain a global $v$ with $\DeltaE v = g$ decaying like $f$. The Newton potential $w = \Gamma * g$ then decays as asserted. By the assumed growth bound on $u$, the harmonic remainder $h = v - w$ has at most linear growth, hence is affine by Liouville. The asymptotic expansion of $u$ follows by restricting back to $\Omega_{2R}$. With the extra regularity, a rescaling and interior Schauder estimates give the derivative bounds.
	
	\begin{lemma}\label{lem:poisson-exterior-affine-sigma}
		Let $\Omega_R:=\R^4\setminus B_R$, and let $u:\Omega_R\to\R^m$ be smooth. Assume $u$ has at most linear growth, i.e. $|u(x)|\le C(1+|x|)$ for $|x|\gg 1$, and
		\begin{align*}
			\Delta_{\gE}u=f,
			\qquad
			f=\mathcal O_\infty(|x|^{-3-\sigma})
			\qquad\text{on }\Omega_R
		\end{align*}
		for some $\sigma\ge 0$. Then there exist $L\in\Hom(\R^4,\R^m)$ and $c\in\R^m$
		such that
		\begin{align*}
			u(x)=Lx+c+E(x),
		\end{align*}
		where
		\begin{align*}
			E(x)
			=
			\begin{cases}
				\mathcal O_\infty(|x|^{-1-\sigma}), & 0\le \sigma<1,\\
				\mathcal O_\infty(|x|^{-2}\log|x|), & \sigma=1,\\
				\mathcal O_\infty(|x|^{-2}), & \sigma>1.
			\end{cases}
		\end{align*}
	\end{lemma}
	
	\medskip
	
	We now turn to the transition map $F=y\circ x^{-1}:\Omega_R\to\Omega_{R'}$. We will establish an estimate for $D^2F$, this serves as the second analytic input required for the proof of Theorem~\ref{thm:main-rigid-bootstrapped}.
	
	Since $F$ is a local isometry between $g^{(x)}$ and $g^{(y)}$, its components satisfy the harmonic map system
	\begin{align}\label{eq:equation-for-F^a}
		(g^{(x)})^{ij}\Bigl(
		\partial_i\partial_j F^a
		-\Gamma^k_{ij}\partial_kF^a
		+\widetilde\Gamma^{a}_{bc}(F)\,\partial_iF^b\,\partial_jF^c
		\Bigr)=0.
	\end{align}
	The asymptotics \eqref{eq:gx-gy-decay-assumption-SR} imply that, after rescaling on fixed annuli,
	this becomes a uniformly elliptic system whose coefficients are close to the Euclidean ones.
	Moreover, the metric equivalence $g^{(x)}\sim \gE$ and $g^{(y)}\sim \gE$ gives $|DF|\le C$ and $|F(x)|\simeq |x|$ for $|x|\gg1$. The next lemma gives the bound for $D^2F$.
	
	\begin{lemma}\label{lem:D2F-decay-estimate}
		In the situation of Theorem~\ref{thm:main-rigid-bootstrapped}, there exist $R_1\ge R$ and $C>0$ such that
		\begin{align}\label{eq:D2F-crude-lemma}
			|D^2F(x)|\le C\,|x|^{-1}
		\end{align}
		for all $x\in\Omega_{R_1}$.
	\end{lemma}
	
	\begin{proof}
		The argument is a standard rescaling argument on a fixed annulus. We zoom in at the scale
		$r_0:=|x_0|$, rewrite the equation for $F$ on a unit-size annulus, and then use interior elliptic estimates
		to obtain a uniform $C^{2,\alpha}$ bound at that scale.
		
		Fix $x_0\in\Omega_{R_1}$ with $r_0:=|x_0|$ large, and set
		\begin{align*}
			F_{r_0}(z):=\frac1{r_0}F(r_0 z)
		\end{align*}
		on the fixed annulus $A:=B_2\setminus B_{1/2}$. Then
		\begin{align*}
			|D_x^2F(x_0)|=r_0^{-1}|D_z^2F_{r_0}(z_0)|,
			\qquad
			z_0:=x_0/r_0\in A.
		\end{align*}
		So it suffices to bound $D_z^2F_{r_0}$ uniformly on a smaller annulus.
		
		Because $|F(x)|\simeq |x|$ and $|DF|\le C$ for $|x|\gg1$, the rescaled maps $F_{r_0}$ are uniformly bounded in
		$C^0$ and $C^1$ on $A$. Rescaling the harmonic map equation \eqref{eq:equation-for-F^a}, we obtain
		\begin{align*}
			a^{ij}(r_0z)\,\partial_{z_i}\partial_{z_j}F_{r_0}^a
			=
			B_{r_0}^k(z)\,\partial_{z_k}F_{r_0}^a
			-
			C^a_{r_0,ij,bc}(z)\,\partial_{z_i}F_{r_0}^b\,\partial_{z_j}F_{r_0}^c
		\end{align*}
		on $A$, where $a^{ij}(r_0z)$ is uniformly elliptic. Moreover, by the decay of the metric coefficients and Christoffel symbols, together with the chain rule,
		the rescaled coefficients satisfy
		\begin{align*}
			\|a^{ij}(r_0\cdot)-\delta^{ij}\|_{C^{0,\alpha}(A)}
			\le C r_0^{-2},
			\quad
			\|B_{r_0}\|_{C^{0,\alpha}(A)}
			+
			\|C_{r_0}\|_{C^{0,\alpha}(A)}
			\le C r_0^{-2}.
		\end{align*}
		In particular, the system is uniformly elliptic on $A$ for $r_0$ large, with coefficient bounds independent of $r_0$.
		
		Now choose nested annuli $A'\Subset A''\Subset A$. Applying the interior $W^{2,p}$ estimate on $A''\Subset A$,
		we obtain
		\begin{align*}
			\|F_{r_0}\|_{W^{2,p}(A'')}\le C_p.
		\end{align*}
		Then Sobolev--Morrey gives a uniform $C^{1,\beta}$ bound on $A''$ for some $\beta\in(0,1)$, so the right-hand side
		of the rescaled system is uniformly bounded in $C^{0,\alpha}(A'')$. A further application of the interior Schauder
		estimate on $A'\Subset A''$ yields
		\begin{align*}
			\|F_{r_0}\|_{C^{2,\alpha}(A')}\le C,
		\end{align*}
		with $C$ independent of $r_0$.
		
		Since $z_0\in A'$, we have $|D_z^2F_{r_0}(z_0)|\le C$, and scaling back gives
		\begin{align*}
			|D_x^2F(x_0)|\le C r_0^{-1}.
		\end{align*}
		This proves \eqref{eq:D2F-crude-lemma}.
	\end{proof}
	
	\begin{proof}[Proof of Theorem~\ref{thm:main-rigid-bootstrapped}]
		We divide the proof into three steps.
		
		\medskip
		\noindent\textbf{Step 1. A first affine approximation.}
		Each component $F^a$ satisfies the harmonic map equation \eqref{eq:equation-for-F^a}, rewriting it with respect to $\Delta_{\gE}$ gives
		\begin{align}\label{eq:laplacian-rewrite-Fa}
			\Delta_{\gE}F^a
			=
			\bigl(\delta^{ij}-(g^{(x)})^{ij}\bigr)\partial_i\partial_jF^a
			+
			(g^{(x)})^{ij}\Gamma(g^{(x)})^k_{ij}\partial_kF^a
			-
			(g^{(x)})^{ij}\widetilde\Gamma(g^{(y)})^a_{bc}(F)\partial_iF^b\partial_jF^c.
		\end{align}
		From \eqref{eq:gx-gy-decay-assumption-SR}, the Christoffel symbols of $g^{(x)}$ and $g^{(y)}$ are
		$\mathcal O_\infty(r^{-3})$.
		Together with Lemma~\ref{lem:D2F-decay-estimate} and $|DF|\le C$, this yields
		\begin{align*}
			\Delta_{\gE}F^a=\mathcal O_\infty(|x|^{-3}).
		\end{align*}
		Applying Lemma~\ref{lem:poisson-exterior-affine-sigma} with $\sigma=0$ componentwise, we obtain
		\begin{align}\label{eq:F-affine-first}
			F(x)=Ax+b+\mathcal E_1(x),
			\qquad
			\mathcal E_1(x)=\mathcal O_\infty(|x|^{-1}),
		\end{align}
		for constant matrix $A$ and constant vector $b$. In particular,
		\begin{align}\label{eq:SR-first-affine-derivative}
			\partial_iF^a=A_{ai}+\mathcal O_\infty(|x|^{-2}).
		\end{align}
		
		Since $F$ is a local isometry, and
		\begin{align*}
			g^{(x)}_{ij}(x)
			=
			g^{(y)}_{ab}(F(x))\,\partial_iF^a\,\partial_jF^b.
		\end{align*}
		Using \eqref{eq:gx-gy-decay-assumption-SR}, \eqref{eq:F-affine-first}, and \eqref{eq:SR-first-affine-derivative},
		and letting $|x|\to\infty$, we obtain
		\begin{align*}
			\delta_{ij}=\delta_{ab}A_{ai}A_{bj},
		\end{align*}
		so $A\in \mathrm O(4)$.
		
		\medskip
		\noindent\textbf{Step 2. Covariance of the Schwarzschild background.}
		
		Denote the linear map $x\to Ax$ by $L$. Since $A\in \mathrm O(4)$, one has $|Ax|=|x|$. Therefore, with $U(z):=1+\tfrac{\lambda}{18|z|^2}$ and $\Phi^\Sch(z):=z-\tfrac{\lambda}{9}\frac{z}{|z|^2}$, we have
		\begin{align*}
			U(Ax)=U(x),
			\qquad
			\Phi^\Sch(Ax)=A\,\Phi^\Sch(x).
		\end{align*}
		Then
		\begin{align}\label{eq:rigid-Sch-background-cov}
			\begin{split}
				L^*g^{\Sch,\sharp}_\lambda
				&=
				L^*\Bigl((\Phi^\Sch)^*(U^2\gE)\Bigr)
				=
				(\Phi^\Sch\circ L)^*(U^2\gE)\\
				&=
				(L\circ \Phi^\Sch)^*(U^2\gE)
				=
				(\Phi^\Sch)^*\bigl(L^*(U^2\gE)\bigr)
				=
				(\Phi^\Sch)^*(U^2\gE)
				=
				g^{\Sch,\sharp}.
			\end{split}
		\end{align}
		
		Now define
		\begin{align*}
			T^{(x),k}(x)
			&:=
			\bigl(g^{\Sch,\sharp}_{x}\bigr)^{ij}
			\Gamma\bigl(g^{\Sch,\sharp}_{x}\bigr)^k_{ij}(x),
			\\
			T^{(y),a}(y)
			&:=
			\bigl(g^{\Sch,\sharp}_{y}\bigr)^{bc}
			\widetilde\Gamma\bigl(g^{\Sch,\sharp}_{y}\bigr)^a_{bc}(y),
		\end{align*}
		and the expansion of reference metric \eqref{eq:gSch-sharp-expansion}  reads
		\begin{align*}
			T^{(x)}=\mathcal O_\infty(|x|^{-3}),
			\quad
			\nabla T^{(x)}=\mathcal O_\infty(|x|^{-4}),
			\quad
			T^{(y)}=\mathcal O_\infty(|y|^{-3}),
			\quad
			\nabla T^{(y)}=\mathcal O_\infty(|y|^{-4}).
		\end{align*}
		
		Moreover, the standard formula for Christoffel symbols under coordinate transformations reads
		\begin{align*}
			\Gamma(g^{\Sch,\sharp}_{x})^k_{ij}(x)
			=
			(A^{-1})^k_a\,A_{bi}A_{cj}\,
			\widetilde\Gamma(g^{\Sch,\sharp}_{y})^a_{bc}(Ax).
		\end{align*}
		Together with the transformation rule for the metric in \eqref{eq:rigid-Sch-background-cov}, we carry the computations as follows
		\begin{align*}
			\bigl(g^{\Sch,\sharp}_{x}\bigr)^{ij}
			\Gamma(g^{\Sch,\sharp}_{x})^k_{ij}(x)\,A_{ak}
			&=
			\Bigl((A^{-1})^i_b(A^{-1})^j_c\,\bigl(g^{\Sch,\sharp}_{y}\bigr)^{bc}(Ax)\Bigr)
			\Bigl((A^{-1})^k_a\,A_{bi}A_{cj}\,
			\widetilde\Gamma(g^{\Sch,\sharp}_{y})^a_{bc}(Ax)\Bigr)\,A_{ak}
			\\
			&=
			(\bigl(g^{\Sch,\sharp}_{y}\bigr))^{bc}(Ax)\,
			\widetilde\Gamma\bigl(g^{\Sch,\sharp}_{y}\bigr)^a_{bc}(Ax),
		\end{align*}
		that is
		\begin{align}\label{eq:rigid-T-covariance}
			T^{(x),k}(x)\,A_{ak}=T^{(y),a}(Ax).
		\end{align}
		
		\medskip
		\noindent\textbf{Step 3. Improvement to $\Delta_{\gE}F=\mathcal O_\infty(|x|^{-4})$.}
		
		Because $x$ and $y$ are preferred coordinate systems, the relative gauge conditions give
		\begin{align*}
			(g^{(x)})^{ij}\Gamma(g^{(x)})^k_{ij}
			=
			(g^{(x)})^{ij}\Gamma(g^{\Sch,\sharp}_{x})^k_{ij},
			\\
			(g^{(y)})^{bc}\widetilde\Gamma(g^{(y)})^a_{bc}
			=
			(g^{(y)})^{bc}\widetilde\Gamma(g^{\Sch,\sharp}_{y})^a_{bc}.
		\end{align*}
		Since
		\begin{align*}
			g^{(x)}-g^{\Sch,\sharp}_{x}
			=
			\mathcal O_\infty(|x|^{-2}),
			\qquad
			g^{(y)}-g^{\Sch,\sharp}_{y}
			=
			\mathcal O_\infty(|y|^{-2}),
		\end{align*}
		and $\Gamma(g^{\Sch,\sharp}_{\bullet})=\mathcal O_\infty(r^{-3})$, we get
		\begin{align}\label{eq:rigid-gauge-replacement}
			(g^{(x)})^{ij}\Gamma(g^{(x)})^k_{ij}
			=
			T^{(x),k}+\mathcal O_\infty(|x|^{-5}),
			\qquad
			(g^{(y)})^{bc}\widetilde\Gamma(g^{(y)})^a_{bc}
			=
			T^{(y),a}+\mathcal O_\infty(|y|^{-5}).
		\end{align}
		
		On the other hand, since $F^*g^{(y)}=g^{(x)}$, we have the exact identity
		\begin{align*}
			(g^{(x)})^{ij}\,\partial_iF^b\,\partial_jF^c
			=
			(g^{(y)})^{bc}(F(x)).
		\end{align*}
		Therefore the last term in \eqref{eq:laplacian-rewrite-Fa} satisfies
		\begin{align*}
			(g^{(x)})^{ij}\widetilde\Gamma(g^{(y)})^a_{bc}(F)\,\partial_iF^b\,\partial_jF^c
			=
			(g^{(y)})^{bc}(F)\,\widetilde\Gamma(g^{(y)})^a_{bc}(F)
			=
			T^{(y),a}(F(x))+\mathcal O_\infty(|x|^{-5}),
		\end{align*}
		where we also used \eqref{eq:rigid-gauge-replacement} in the $y$-chart.
		
		Substituting these facts into \eqref{eq:laplacian-rewrite-Fa}, and using
		\eqref{eq:SR-first-affine-derivative}, we obtain
		\begin{align*}
			\Delta_{\gE}F^a
			&=
			\bigl(\delta^{ij}-(g^{(x)})^{ij}\bigr)\partial_i\partial_jF^a
			+
			T^{(x),k}\partial_kF^a
			-
			T^{(y),a}(F(x))
			+
			\mathcal O_\infty(|x|^{-5})\\
			&=
			T^{(x),k}A_{ak}
			-
			T^{(y),a}(F(x))
			+
			\mathcal O_\infty(|x|^{-5}).
		\end{align*}
		Furthermore, by Taylor expansion,
		\begin{align*}
			T^{(y),a}(F(x))
			=
			T^{(y),a}(Ax)+\mathcal O_\infty(|x|^{-4}).
		\end{align*}
		Combining this with \eqref{eq:rigid-T-covariance}, we conclude
		\begin{align}\label{eq:rigid-DeltaF-improved}
			\Delta_{\gE}F^a
			=
			\mathcal O_\infty(|x|^{-4}).
		\end{align}
		
		\medskip
		\noindent\textbf{Step 4. Final asymptotics and the $\Gamma$-statement.}
		
		Set $G(x):=F(x)-Ax-b$. Then $G=\mathcal O_\infty(|x|^{-1})$ by \eqref{eq:F-affine-first}, and
		\eqref{eq:rigid-DeltaF-improved} gives
		\begin{align*}
			\Delta_{\gE}G=\mathcal O_\infty(|x|^{-4}).
		\end{align*}
		Applying Lemma~\ref{lem:poisson-exterior-affine-sigma} to $G$ with $\sigma=1$,
		we obtain an expansion
		\begin{align*}
			G(x)=Lx+c+\mathcal E(x),
			\quad\mathcal E(x)=\mathcal O_\infty(|x|^{-2}\log|x|).
		\end{align*}
		But $G(x)=\mathcal O_\infty(|x|^{-1})$ already, so necessarily $L=0$ and $c=0$.
		Hence
		\begin{align*}
			F(x)=Ax+b+\mathcal E(x),
			\qquad
			\mathcal E(x)=\mathcal O_\infty(|x|^{-2}\log|x|),
		\end{align*}
		and, by the derivative part of Lemma~\ref{lem:poisson-exterior-affine-sigma},
		\begin{align*}
			\partial^{(k)}\mathcal E(x)
			=
			\mathcal O_\infty(|x|^{-2-k}\log|x|)
			\qquad(k\ge 1).
		\end{align*}
		This proves \eqref{eq:F-affine-bootstrapped-SR} and
		\eqref{eq:E-bootstrapped-decay-SR}.
		
		Finally, assume both charts are $\Gamma$-compatible. For each $\gamma\in\Gamma$, the transition map sends
		$\gamma$-equivariant $x$-coordinates to $\gamma$-equivariant $y$-coordinates, so there exists
		$\widehat\gamma\in\Gamma$ such that $F(\gamma x)=\widehat\gamma F(x)$ for $|x|\gg1$. Inserting \eqref{eq:F-affine-bootstrapped-SR} and comparing the linear terms gives $A\gamma=\widehat\gamma A$, thus $A$ normalizes $\Gamma$.
	\end{proof}
	
	With Theorem~\ref{thm:main-rigid-bootstrapped} in hand, the proof of the second main theorem, Theorem~\ref{thm:main2}, is straightforward. We next compute the transformation law of the canonical coefficient $A^\sharp$ and of the splitting data.
	
	\begin{proof}[Proof of Theorem~\ref{thm:main2}]
		The coordinate change identity is
		\begin{align}
			g^{(x)}_{ij}(x)
			=
			\partial_iF^a(x)\,\partial_jF^b(x)\,g^{(y)}_{ab}(F(x)).
			\label{eq:coordinate-change-identity}
		\end{align}
		From Theorem~\ref{thm:main-rigid-bootstrapped}, (we still denote $r=|x|$)
		\begin{align}\label{eq:partial_Fa}
			\partial_iF^a(x)=A_{ai}+\mathcal O_\infty(r^{-3}\log r),
			\quad
			F(x)=Ax+b+\mathcal O_\infty(r^{-2}\log r),
		\end{align}
		Write $y=F(x)$ and $\rho=|y|$.
		Since $A\in \mathrm O(4)$ and $b$ is constant,
		\begin{align*}
			\rho^{-4}=r^{-4}+\mathcal O(r^{-5}),
			\quad
			\rho^{-2-\sigma}=r^{-2-\sigma}+\mathcal O(r^{-3-\sigma}),
		\end{align*}
		also
		\begin{align*}
			y_py_q
			=
			(Ax)_p(Ax)_q+\mathcal O(r).
		\end{align*}
		Therefore,
		\begin{align}\label{eq:yyrho4-expansion}
			\frac{y_py_q}{\rho^4}
			=
			\frac{(Ax)_p(Ax)_q}{r^4}+\mathcal{O}_\infty(r^{-3})
			=
			A_{pk}A_{q\ell}\,\frac{x_kx_\ell}{r^4}+\mathcal{O}_\infty(r^{-3}).
		\end{align}
		Substituting \eqref{eq:gx-gy-decay-assumption-SR}, \eqref{eq:partial_Fa}, and \eqref{eq:yyrho4-expansion}
		into \eqref{eq:coordinate-change-identity}, we obtain
		\begin{align*}
			g^{(x)}_{ij}(x)
			=
			\delta_{ij}
			+
			\Bigl(A_{ai}A_{bj}A_{pk}A_{q\ell}\,A^{\sharp,(y)}_{abpq}\Bigr)\frac{x_kx_\ell}{r^4}
			+
			\mathcal O_\infty(r^{-2-\sigma}).
		\end{align*}
		Comparing with the $x$-expansion in \eqref{eq:gx-gy-decay-assumption-SR} gives
		\begin{align}\label{eq:Asharp-index-relation}
			A^{\sharp,(x)}_{ijk\ell}
			=
			A_{ai}A_{bj}A_{pk}A_{q\ell}\,A^{\sharp,(y)}_{abpq}.
		\end{align}
		This is instrumental in identifying how the Weyl-type part behaves under transformations between preferred coordinates. 
		
		Rearrange the splitting of $A^\sharp$ in each chart as (recall $\lambda=9|\Gamma|\,m_{\ADM}(g)$)
		\begin{align*}
			\begin{split}
				s\bigl(\Weyl_\infty^{(x)}\bigr)_{ijk\ell}
				&=
				A^{\sharp,(x)}_{ijk\ell}
				-\Xi\bigl(\lambda\bigr)_{ijk\ell},
				\\
				s\bigl(\Weyl_\infty^{(y)}\bigr)_{abpq}
				&=
				A^{\sharp,(y)}_{abpq}
				-\Xi\bigl(\lambda\bigr)_{abpq}.
			\end{split}
		\end{align*}
		The operator $\Xi$ is $\mathrm O(4)$-equivariant in the sense that
		\begin{align*}
			\Xi(\lambda)_{ijk\ell}
			&=
			A_{ai}A_{bj}A_{pk}A_{q\ell}\,\Xi(\lambda)_{abpq},
		\end{align*}
		and, for any $(0,4)$-tensor $R$,
		\begin{align*}
			s(R)_{ijk\ell}=\Bigl(\frac12(R_{ik\ell j}+R_{i\ell kj})\Bigr)
			=
			A_{ai}A_{pk}A_{q\ell}A_{bj}\,
			\Bigl(\frac12(\widetilde R_{apqb}+\widetilde R_{aqpb})\Bigr)
			=A_{ai}A_{pk}A_{q\ell}A_{bj}
			s(\widetilde R)_{abpq},
		\end{align*}
		whenever
		\begin{align*}
			R_{ik\ell j}=A_{ai}A_{pk}A_{q\ell}A_{bj}\,\widetilde R_{apqb}.
		\end{align*}
		Combining these properties with \eqref{eq:Asharp-index-relation}, we compute
		\begin{align*}
			s(\Weyl_\infty^{(x)})_{ijk\ell}
			&=A_{ai}A_{bj}A_{pk}A_{q\ell}\left( A^{\sharp,(y)}_{abpq}
			-\Xi(\lambda)_{abpq}\right) 
			\\
			&=A_{ai}A_{bj}A_{pk}A_{q\ell}\,
			s\bigl(\Weyl_\infty^{(y)}\bigr)_{abpq}
			\\
			&=s\left( \Weyl^{\prime}\right)_{ijk\ell},
		\end{align*}
		where $\Weyl^{\prime}$ is defined as $\Weyl^{\prime}_{ijk\ell}=A_{ai}A_{bj}A_{pk}A_{q\ell}\bigl(\Weyl_\infty^{(y)}\bigr)_{abpq}$. Since $s$ is injective on algebraic curvature tensors, we conclude
		\begin{align*}
			(\Weyl_\infty^{(x)})_{ijk\ell}
			=
			A_{ai}A_{bj}A_{ck}A_{d\ell}\,(\Weyl_\infty^{(y)})_{abcd}.
		\end{align*}
		This proves Theorem~\ref{thm:main2}.
	\end{proof}
	
	\subsection{An example: the Burns metric}\label{subsec:burns-example}
	
	We now present the detailed computation for the Burns metric, announced in the introduction, to illustrate the statements of Theorem~\ref{thm:main1} and Theorem~\ref{thm:main2}. We work in the standard ALE coordinates on the end of $(\mathrm{Bl}_0\mathbb C^2,g^{\Burns})$, where the group at infinity is trivial. Starting from the K\"ahler potential, we compute the rough
	$r^{-2}$ coefficient $A^{\Burns}$, then, just as we previously tracked the Weyl-type tensor from the rough coefficient, we will perform the corresponding coordinate transformation, and finally obtain
	the splitting
	\begin{align*}
		(A^{\Burns})^\sharp=\Xi(3)+s(\Weyl_\infty^{\mathrm{Burns}}),
		\qquad
		\Weyl_\infty^{\mathrm{Burns}}\neq 0 .
	\end{align*}
	This shows concretely that the mass part and the Weyl-type part in
	Theorem~\ref{thm:main1} may both be nonzero.
	
	On the end, we identify $\mathbb C^{2}\setminus\{0\}$ with $\mathbb R^{4}\setminus\{0\}$ and write
	$r=|x|$.  We use the convention under which the flat K\"ahler potential
	$s=|z|^{2}$ gives the standard Euclidean metric $g_E$ on
	$\mathbb C^{2}\simeq\mathbb R^{4}$.  With this convention, the Burns metric is
	given by the K\"ahler potential
	$\Phi(s)=s+\log s$, where $s=|z|^{2}=r^{2}$.  Hence
	\begin{align*}
		g_{a\bar b}^{\Burns}
		=
		\partial_a\partial_{\bar b}\Phi
		=
		\left(1+\frac1{s}\right)\delta_{ab}
		-
		\frac{\bar z_a z_b}{s^2}.
	\end{align*}
	Let $J$ be the standard complex structure on
	$\mathbb R^{4}\simeq\mathbb C^{2}$, and write
	$(Jx)_i=J_{ik}x_k$, where the index of $J$ is lowered using the Euclidean
	metric.  In real coordinates, the preceding identity gives the exact formula
	\begin{align}\label{eq:Burns-real-expansion2}
		g_{ij}^{\Burns}
		=
		\delta_{ij}
		+
		\frac1{r^2}\delta_{ij}
		-
		\frac{x_i x_j}{r^4}
		-
		\frac{(Jx)_i(Jx)_j}{r^4}.
	\end{align}
	Thus $g^{\Burns}=g_E+h^{\Burns}$ on the end, where the order $r^{-2}$ perturbation has the form
	\begin{align}\label{eq:Burns-A}
		h_{ij}^{\Burns}
		=
		A_{ijk\ell}^{\Burns}\frac{x_kx_\ell}{r^4},
		\quad 
		A_{ijk\ell}^{\Burns}
		=
		\delta_{ij}\delta_{k\ell}
		-
		\frac12
		\bigl(
		\delta_{ik}\delta_{j\ell}
		+
		\delta_{i\ell}\delta_{jk}
		\bigr)
		-
		\frac12
		\bigl(
		J_{ik}J_{j\ell}
		+
		J_{i\ell}J_{jk}
		\bigr).
	\end{align}
	
	A natural next step is to compute the two partial traces of the rough coefficient $A^{\Burns}$ and the ADM mass of the metric $g^{\Burns}$. Still set
	$U_{k\ell}^{\Burns}:=A_{ppk\ell}^{\Burns}$ and $V_{i\ell}^{\Burns}:=A_{ijj\ell}^{\Burns}$.  Since $J$ is an
	orthogonal complex structure, there holds $J_{pk}J_{p\ell}=\delta_{k\ell}$ and $J_{ij}J_{j\ell}=-\delta_{i\ell}$. Then we have
	\begin{align*}
		U_{k\ell}^{\Burns}=2\delta_{k\ell},
		\qquad
		V_{i\ell}^{\Burns}=-\delta_{i\ell},
	\end{align*}
	and
	\begin{align*}
		\lambda^{\Burns}(A)
		=
		\frac14\bigl(\tr(V^{\Burns})+2\tr(U^{\Burns})\bigr)
		=
		3.
	\end{align*}
	Since the group at infinity is trivial, Theorem~\ref{thm:main1} predicts that
	this scalar quantity $\lambda^{\Burns}(A)$ should be $9m_{\ADM}(g^{\Burns})$, which means the Burns metric should satisfy
	\begin{align}\label{eq:Burns-mass}
		m_{\ADM}(g^{\Burns})=\frac13.
	\end{align}
	This value agrees with the standard computation of the ADM mass of the Burns
	metric in K\"ahler geometry.  Indeed, the Burns metric with potential
	$\Phi(s)=s+\log s$ has ADM mass $1/3$ under the normalization used here; this is
	consistent with the general ALE K\"ahler mass formula of Hein--LeBrun
	\cite{HL}.
	
	Now we carry out the coordinate transformation analogous to the general case; see Proposition~\ref{prop:W-properties} for details. Consider the following coordinate transformation:
	\begin{align}\label{eq:coordinate-transformation-Burns}
		x_i
		=
		y_i
		+
		B_{ij}\frac{y_j}{|y|^2},
		\quad
		B=-\frac12 I.
	\end{align}
	At the level of the $r^{-2}$ coefficient, this changes $A^{\Burns}$ into
	$(A^{\Burns})^\sharp=A^{\Burns}+\Psi_4(B)$. For this choice of $B$, the definition of $\Psi_4$ (c.f. \eqref{eq:Psi4-def}) gives
	\begin{align*}
		\Psi_4(B)_{ijk\ell}
		=
		-\delta_{ij}\delta_{k\ell}
		+
		\delta_{ik}\delta_{j\ell}
		+
		\delta_{i\ell}\delta_{jk}.
	\end{align*}
	Combining this with \eqref{eq:Burns-A}, we get
	\begin{align*}
		(A^{\Burns})^\sharp_{ijk\ell}
		=
		\frac12
		\bigl(
		\delta_{ik}\delta_{j\ell}
		+
		\delta_{i\ell}\delta_{jk}
		\bigr)
		-
		\frac12
		\bigl(
		J_{ik}J_{j\ell}
		+
		J_{i\ell}J_{jk}
		\bigr).
	\end{align*}
	
	Analogous to the splitting we performed on $A^\sharp$ in \eqref{eq:Asharp-splitting-version-2}, we may separate the mass part of $A^{\Burns}$. Since $\lambda(A^{\Burns})=3$, the corresponding scalar
	tensor is
	\begin{align*}
		\Xi(3)_{ijk\ell}
		=
		-\frac13\delta_{ij}\delta_{k\ell}
		+
		\frac23
		\bigl(
		\delta_{ik}\delta_{j\ell}
		+
		\delta_{i\ell}\delta_{jk}
		\bigr).
	\end{align*}
	Hence the remaining part is
	\begin{align*}
		\begin{split}
			W(A^{\Burns})_{ijk\ell}
			&:=
			(A^{\Burns})^\sharp_{ijk\ell}-\Xi(3)_{ijk\ell}
			\\
			&=
			\frac13\delta_{ij}\delta_{k\ell}
			-
			\frac16
			\bigl(
			\delta_{ik}\delta_{j\ell}
			+
			\delta_{i\ell}\delta_{jk}
			\bigr)
			-
			\frac12
			\bigl(
			J_{ik}J_{j\ell}
			+
			J_{i\ell}J_{jk}
			\bigr).
		\end{split}
	\end{align*}
	This tensor is not zero.  For example, if the orthonormal basis is chosen so
	that $J_{12}J_{21}=-1$, then $W(A^{\Burns})_{1212}=\frac13$. On the other hand, $W(A^{\Burns})$ has exactly the algebraic properties expected from
	the Weyl-type part. Direct contraction shows $W(A^{\Burns})_{ppk\ell}=0$ and $W(A^{\Burns})_{ijj\ell}=0$ hold. Moreover, using $J_{ij}=-J_{ji}$ and $J^2=-I$, one checks that $W(A^{\Burns})$ satisfies the first Bianchi identity. Therefore $W(A^{\Burns})$ lies in the image of the map $s$ applied to algebraic Weyl
	tensors by Proposition~\ref{prop:Wtilde-eq}. Thus there exists a nonzero algebraic Weyl tensor
	$\Weyl_\infty^{\mathrm{Burns}}$ such that
	\begin{align*}
		W(A^{\Burns})=s\bigl(\Weyl_\infty^{\mathrm{Burns}}\bigr).
	\end{align*}
	
	Consequently, after applying the coordinate transformation \eqref{eq:coordinate-transformation-Burns}, the Burns metric \eqref{eq:Burns-real-expansion2} has
	the expansion
	\begin{align*}
		g_{ij}^{\Burns}
		=
		\delta_{ij}
		+
		\left(
		\Xi(3)_{ijk\ell}
		+
		s\bigl(\Weyl_\infty^{\mathrm{Burns}}\bigr)_{ijk\ell}
		\right)
		\frac{y_ky_\ell}{|y|^4}
		+
		\mathcal O_\infty(|y|^{-3}).
	\end{align*}
	Since $\Gamma=\{1\}$ and $3=9m_{\ADM}(g^{\Burns})$ by \eqref{eq:Burns-mass}, this is
	precisely the form predicted by Theorem~\ref{thm:main1}. The subsequent gauge correction used to impose the relative Schwarzschild gauge
	does not change this displayed $r^{-2}$ coefficient.  Hence the Burns metric
	gives a concrete example in which the leading term contains both a nonzero mass part and a nonzero Weyl-type part. Furthermore, Theorem~\ref{thm:main2} can also be checked directly from this explicit coefficient.
	
	\section{Leading Weyl data and crepancy for minimal quotient resolutions}
	\label{sec:Kahler-ALE-minimal-resolution}
	
	In this section we specialize the preferred asymptotic expansion obtained above
	to the scalar-flat K\"ahler ALE metrics on minimal resolutions of quotient
	surface singularities.  Let
	$\Gamma\subset \mathrm U(2)$ be a finite subgroup acting freely on $\Sph^3$,
	set $Y=\mathbb C^2/\Gamma$, and let
	\begin{align*}
		\pi:X\longrightarrow Y
	\end{align*}
	be the minimal resolution.  The resolution map identifies
	$X\setminus \pi^{-1}(0)$ with the smooth part of the quotient, and therefore
	gives $X$ a natural holomorphic ALE end.
	
	Scalar-flat K\"ahler ALE metrics on such spaces are provided by the existence
	results of Calderbank--Singer in the cyclic case and Lock--Viaclovsky in the
	non-cyclic case (c.f. \cite{CS04},\cite{LV19}).  Our aim here is to use the leading asymptotic invariant
	constructed in Theorem~\ref{thm:main1} in the reverse direction: starting from a
	scalar-flat K\"ahler ALE metric on $X$, we ask what the Weyl tensor at infinity
	remembers about the resolution $\pi$.
	
	The first step is to compute the preferred
	$|x|^{-2}$ coefficient in the K\"ahler setting and relate its Weyl part to the
	ADM mass and the limiting complex structure at infinity.  The second step uses
	the Hein--LeBrun mass formula to rewrite this coefficient in terms of the
	exceptional curves of the minimal resolution.  This will give the crepancy
	criterion in Theorem~\ref{thm:main3}.
	
	\subsection{The K\"ahler ALE coefficient at infinity}
	\label{subsec:Kahler-ALE-coefficient}
	
	Let $g$ be a complete scalar-flat K\"ahler ALE metric on $(X,J)$, and let
	$\omega=g(J\cdot,\cdot)$ be its K\"ahler form.  On the quotient end described
	above, lift to the $\Gamma$-cover and use the holomorphic coordinates
	$z=(z^1,z^2)$ coming from the quotient map.  Let $u=(u_1,\dots,u_4)$ be the
	associated real coordinates and put $r=|u|=|z|$.  We first compute the
	homogeneous $r^{-2}$ coefficient in these holomorphic coordinates.
	
	In these coordinates, the scalar-flat K\"ahler ALE potential has the form
	\begin{align*}
		\omega
		=
		\sqrt{-1}\,\del\bar\del
		\bigl(
		\rho+c\log\rho+\psi
		\bigr),
		\qquad
		\rho=|z|^2=r^2,
		\qquad
		\psi=\mathcal O_\infty(r^{-1}),
	\end{align*}
	for some real constant $c$.  We use the convention that the flat potential
	$\rho=|z|^2$ gives the Euclidean metric.  Such expansions are standard in the
	scalar-flat K\"ahler ALE setting; see for example \cite{HL,AC21,HRS}.  The
	constant $c$ denotes the logarithmic coefficient, and should not be confused with
	the discrepancy coefficients used later.
	
	Let $J_\infty$ denote the standard limiting complex structure in the holomorphic
	coordinates $u$.  If $F=F(\rho)$ is a radial K\"ahler potential, then the
	associated real metric is
	\begin{align*}
		(g_F)_{ij}
		=
		F'(\rho)\delta_{ij}
		+
		F''(\rho)
		\Bigl(
		u_i u_j+(J_\infty u)_i(J_\infty u)_j
		\Bigr).
	\end{align*}
	For $F(\rho)=\rho+c\log\rho$, one has $F'(\rho)=1+c/\rho$ and
	$F''(\rho)=-c/\rho^2$.  Since $\psi=\mathcal O_\infty(r^{-1})$, its contribution
	to the metric is $\mathcal O_\infty(r^{-3})$.  Hence
	\begin{align}\label{eq:Kahler-ALE-real-expansion}
		g_{ij}
		=
		\delta_{ij}
		+
		c
		\left(
		\frac{\delta_{ij}}{|u|^2}
		-
		\frac{u_i u_j}{|u|^4}
		-
		\frac{(J_\infty u)_i(J_\infty u)_j}{|u|^4}
		\right)
		+
		\mathcal O_\infty(|u|^{-3}).
	\end{align}
	
	The coefficient $c$ is determined by the ADM mass.  Put $h_{ij}=g_{ij}-\delta_{ij}$.
	From \eqref{eq:Kahler-ALE-real-expansion}, one obtains
	\begin{align*}
		\partial_jh_{ij}-\partial_i h_{jj}
		=
		2c\frac{u_i}{|u|^4}
		+
		\mathcal O_\infty(|u|^{-4}).
	\end{align*}
	Contracting with the Euclidean unit normal on $S_r$ in the cover and using the
	ADM convention \eqref{eq:mass-def-prep}, we get
	\begin{align*}
		c
		=
		3|\Gamma|\,m_{\ADM}(g).
	\end{align*}
	
	The holomorphic ALE coordinates are not yet the preferred coordinates of
	Theorem~\ref{thm:main1}.  At the level of the homogeneous $|u|^{-2}$ coefficient, the required coordinate transformation is
	\begin{align*}
		u_i
		=
		x_i
		-
		\frac{c}{2}\frac{x_i}{|x|^2}.
	\end{align*}
	Equivalently, this is $u_i=x_i+B_{ij}x_j/|x|^2$ with $B=-\frac{c}{2}I$. Using
	\eqref{eq:Psi4-def}, or directly differentiating the coordinate change, the
	metric becomes
	\begin{align}\label{eq:Kahler-ALE-sharp-expansion}
		g^\sharp_{ab}
		=
		\delta_{ab}
		+
		c
		\left(
		\frac{x_a x_b}{|x|^4}
		-
		\frac{(J_\infty x)_a(J_\infty x)_b}{|x|^4}
		\right)
		+
		\mathcal O_\infty(|x|^{-3}).
	\end{align}
	It remains to pass from the intermediate coordinates in
	\eqref{eq:Kahler-ALE-sharp-expansion} to the preferred coordinates of
	Theorem~\ref{thm:main1}.  This last coordinate change only affects lower-order
	terms.  Therefore the homogeneous $|x|^{-2}$ coefficient in
	\eqref{eq:Kahler-ALE-sharp-expansion} is unchanged.  The only remaining freedom is
	the orthogonal ambiguity described in Theorem~\ref{thm:main2}, so we keep the
	coordinate label in the final formula.
	
	Rewrite \eqref{eq:Kahler-ALE-sharp-expansion} as
	\begin{align}\label{eq:preferred-expnasion-Kahler-ALE}
		g^\sharp_{ij}
		&=
		\delta_{ij}
		+
		(A^K)^\sharp_{ijk\ell}
		\frac{x_k x_\ell}{|x|^4}
		+
		\mathcal O_\infty(|x|^{-3}),
	\end{align}
	here
	\begin{align*}
		(A^K)^\sharp_{ijk\ell}
		&=
		\frac{c}{2}
		\bigl(
		\delta_{ik}\delta_{j\ell}
		+
		\delta_{i\ell}\delta_{jk}
		\bigr)
		-
		\frac{c}{2}
		\bigl(
		(J_\infty)_{ik}(J_\infty)_{j\ell}
		+
		(J_\infty)_{i\ell}(J_\infty)_{jk}
		\bigr).
	\end{align*}
	The scalar part in the preferred expansion \eqref{eq:preferred-expnasion-Kahler-ALE} is
	$\Xi(9|\Gamma|m_{\ADM}(g))=\Xi(3c)$ according to \eqref{eq:intro-main-expansion}.  Subtracting this scalar part from
	$(A^K)^\sharp$ gives
	\begin{align*}
		(A^K)^\sharp-\Xi(3c)
		=
		c\,Q_{J_\infty},
	\end{align*}
	where
	\begin{align*}
		(Q_{J_\infty})_{ijk\ell}
		=
		&
		\frac13\delta_{ij}\delta_{k\ell}
		-
		\frac16
		\bigl(
		\delta_{ik}\delta_{j\ell}
		+
		\delta_{i\ell}\delta_{jk}
		\bigr)
		\\
		&-
		\frac12
		\bigl(
		(J_\infty)_{ik}(J_\infty)_{j\ell}
		+
		(J_\infty)_{i\ell}(J_\infty)_{jk}
		\bigr).
	\end{align*}
	Using the algebraic characterization in Proposition~\ref{prop:Wtilde-eq}, one
	checks that $Q_{J_\infty}$ is the image under $s$ of a unique algebraic Weyl
	tensor.  We denote this tensor by $\mathcal W_{J_\infty}$, so that
	\begin{align*}
		s(\mathcal W_{J_\infty})
		=
		Q_{J_\infty}.
	\end{align*}
	Equivalently, it is given by
	\begin{align}\label{eq:Kahler-ALE-WJ-explicit}
		\begin{split}
			(\mathcal W_{J_\infty})_{ijk\ell}
			=
			&
			\frac13
			\bigl(
			\delta_{i\ell}\delta_{jk}
			-
			\delta_{ik}\delta_{j\ell}
			\bigr)
			+
			\frac23
			(J_\infty)_{ij}(J_\infty)_{k\ell}
			\\
			&+
			\frac13
			(J_\infty)_{ik}(J_\infty)_{j\ell}
			-
			\frac13
			(J_\infty)_{i\ell}(J_\infty)_{jk}.
		\end{split}
	\end{align}
	This tensor is nonzero.  For example, in a unitary orthonormal basis one has
	$(Q_{J_\infty})_{1212}=1/3$.
	
	Thus the homogeneous coefficient in the preferred expansion \eqref{eq:preferred-expnasion-Kahler-ALE} satisfies
	\begin{align*}
		(A^K)^\sharp
		&=
		\Xi(3c)
		+
		s(c\mathcal W_{J_\infty})
		\\
		&=\Xi\bigl(9|\Gamma|m_{\ADM}(g)\bigr)
		+
		s\bigl(3|\Gamma|m_{\ADM}(g)\mathcal W_{J_\infty}\bigr).
	\end{align*}
	Comparing with the preferred expansion in \eqref{eq:intro-main-expansion}, and using
	the injectivity of $s$ on algebraic Weyl tensors, gives the leading Weyl term in
	the preferred coordinates:
	\begin{align}\label{eq:Weyl-and-complex-structure}
		\Weyl_\infty(g)=3|\Gamma|m_{\ADM}(g)\mathcal W_{J_\infty}.
	\end{align}
	
	The coordinate covariance of formula \eqref{eq:Weyl-and-complex-structure} is also a necessary consideration. If $A\in\mathrm O(4)$, we regard $J$ as a two-form using the Euclidean metric
	and write
	\begin{align*}
		(A\cdot J)_{ij}
		=
		A_{ai}A_{bj}J_{ab}
	\end{align*}
	for the induced orthogonal action. Recall for an algebraic Weyl tensor $W$, we write
	\begin{align*}
		(A\cdot W)_{ijk\ell}
		=
		A_{ai}A_{bj}A_{ck}A_{d\ell}W_{abcd}.
	\end{align*}
	Then the tensor in \eqref{eq:Kahler-ALE-WJ-explicit} is equivariant under this action:
	\begin{align}\label{eq:Kahler-WJ-O4-equivariance}
		\mathcal W_{A\cdot J}
		=
		A\cdot\mathcal W_J .
	\end{align}
	
	The preceding computation identifies the leading Weyl part in a preferred
	coordinate system obtained from the holomorphic ALE end.  Since
	$\mathcal W_J$ is equivariant under the orthogonal action
	\eqref{eq:Kahler-WJ-O4-equivariance}, this identification is preserved when one
	changes between preferred coordinate systems.  We record the result in the
	following proposition.
	
	\begin{proposition}\label{prop:Kahler-ALE-Weyl-infinity}
		Let $\pi:X\longrightarrow \mathbb C^2/\Gamma$ be the minimal resolution fixed
		above, and let $g$ be a scalar-flat K\"ahler ALE metric on $(X,J)$.  Let $x$
		be a preferred coordinate system obtained by applying Theorem~\ref{thm:main1}
		starting from the holomorphic ALE end induced by $\pi$.  Denote by
		$J_\infty^{(x)}$ the limiting Euclidean complex structure in the $x$-coordinates,
		and denote by $\Weyl_\infty^{(x)}(g)$ the Weyl tensor at infinity in the same
		coordinates.  Then
		\begin{align}\label{eq:Kahler-ALE-Weyl-infinity-formula}
			\Weyl_\infty^{(x)}(g)
			=
			3|\Gamma|\,m_{\ADM}(g)\,\mathcal W_{J_\infty^{(x)}} .
		\end{align}
		
		Moreover, if $x$ and $y$ are two such preferred coordinate systems and $A$ is
		the orthogonal linear part of the transition map at infinity in view of Theorem~\ref{thm:main-rigid-bootstrapped}, then
		\begin{align*}
			J_\infty^{(x)}
			=
			A\cdot J_\infty^{(y)},
			\qquad
			\Weyl_\infty^{(x)}(g)
			=
			A\cdot \Weyl_\infty^{(y)}(g).
		\end{align*}
		Consequently, the formula \eqref{eq:Kahler-ALE-Weyl-infinity-formula} is
		preserved under every change between such preferred coordinate systems.
	\end{proposition}
	
	\begin{proof}
		The formula in one preferred coordinate system follows from the computation
		above.  It remains only to check the behavior under a change of preferred
		coordinates.  By Theorem~\ref{thm:main2}, for the asymptotic expansion of $F = y \circ x^{-1}$, with leading coefficient $A \in \mathrm{O}(4)$, we have
		\begin{align*}
			\Weyl_\infty^{(x)}(g)
			=
			A\cdot \Weyl_\infty^{(y)}(g).
		\end{align*}
		The same transition map $F = y \circ x^{-1}$ sends the limiting complex structure in the
		$y$-coordinates to the limiting complex structure in the $x$-coordinates, so
		$J_\infty^{(x)}=A\cdot J_\infty^{(y)}$. Using
		\eqref{eq:Kahler-WJ-O4-equivariance}, we get
		\begin{align*}
			A\cdot
			\mathcal W_{J_\infty^{(y)}}
			=
			\mathcal W_{A\cdot J_\infty^{(y)}}
			=
			\mathcal W_{J_\infty^{(x)}}.
		\end{align*}
		Therefore, if the formula holds in the $y$-coordinates, then
		\begin{align*}
			\Weyl_\infty^{(x)}(g)
			&=
			A\cdot \Weyl_\infty^{(y)}(g)
			\\
			&=
			A\cdot
			\left(
			3|\Gamma|\,m_{\ADM}(g)\,\mathcal W_{J_\infty^{(y)}}
			\right)
			\\
			&=
			3|\Gamma|\,m_{\ADM}(g)\,\mathcal W_{J_\infty^{(x)}}.
		\end{align*}
		The converse follows by interchanging $x$ and $y$.
	\end{proof}
	
	\begin{remark}
		Formula \eqref{eq:Kahler-ALE-Weyl-infinity-formula} shows that, in the
		scalar-flat K\"ahler ALE setting, the Weyl tensor at infinity is not an
		arbitrary Weyl-type boundary datum.  The K\"ahler structure singles out a
		distinguished algebraic Weyl tensor
		$\mathcal W_{J_\infty^{(x)}}$, and the whole Weyl contribution is a scalar
		multiple of this tensor.  The scalar multiple is exactly
		$3|\Gamma|m_{\ADM}(g)$. Thus the ADM mass determines both the scalar part of the preferred expansion
		and the size of the Weyl-type part, while the limiting complex structure
		determines the tensor direction.
	\end{remark}

	We now replace the mass coefficient in
	\eqref{eq:Kahler-ALE-Weyl-infinity-formula} by a complex-geometric quantity.  For
	a scalar-flat K\"ahler ALE surface, the Hein--LeBrun mass formula gives
	\begin{align}\label{eq:HL-mass-formula-dim2}
		m_{\ADM}(g)
		=
		-\frac{1}{3\pi}
		\left\langle
		\clubsuit(c_1(X,J)),[\omega]
		\right\rangle .
	\end{align}
	Here $\clubsuit$ is the inverse of the natural map from compactly supported
	cohomology to ordinary cohomology \cite{HL}.  In the present ALE setting, this
	map is an isomorphism over $\R$, and the bracket is the Poincar\'e duality
	pairing
	\begin{align*}
		H_c^2(X;\R)\times H^2(X;\R)
		\longrightarrow
		\R .
	\end{align*}
	Combining \eqref{eq:Kahler-ALE-Weyl-infinity-formula} and
	\eqref{eq:HL-mass-formula-dim2} gives the form used below.
	
	\begin{corollary}\label{cor:Kahler-ALE-Weyl-HL}
		In every preferred coordinate system $x$ obtained from the holomorphic end of
		the resolution, one has
		\begin{align*}
			\Weyl_\infty^{(x)}(g)
			=
			-\frac{|\Gamma|}{\pi}
			\left\langle
			\clubsuit(c_1(X,J)),[\omega]
			\right\rangle
			\mathcal W_{J_\infty^{(x)}}.
		\end{align*}
	\end{corollary}
	
	Thus, after the holomorphic end of the minimal resolution is fixed, the
	vanishing of $\Weyl_\infty^{(x)}(g)$ is equivalent to the vanishing of the
	Chern--K\"ahler pairing $\left\langle
	\clubsuit(c_1(X,J)),[\omega]
	\right\rangle$. The vanishing statement is independent of the preferred coordinate system, since
	two representatives of $\Weyl_\infty$ differ by an orthogonal action.  We now
	express this pairing in terms of the exceptional curves of the minimal resolution.
	
	\subsection{Minimal quotient resolutions and crepancy}
	\label{subsec:minimal-quotient-resolution}
	
	We keep the notation fixed above. Thus $\pi:X\to Y=\mathbb C^2/\Gamma$ is the minimal resolution, and $g$ is a scalar-flat K\"ahler ALE metric on
	$(X,J)$ with K\"ahler form $\omega$.  The resolution map identifies the end of
	$X$ with $Y\setminus\{0\}$, so the holomorphic end used in
	Subsection~\ref{subsec:Kahler-ALE-coefficient} is the one determined by $\pi$. In this subsection, we will proof Theorem~\ref{thm:main3}.
	
	Write the exceptional divisor as $E=\bigcup_{i=1}^k E_i$, and let
	$Q=(E_i\cdot E_j)$ be its intersection matrix.  If $\Gamma$ is trivial, then
	$E=\emptyset$ and the sums below are understood to be zero.  Otherwise, for a
	quotient surface singularity, the curves $E_i$ are rational curves, $E$ is
	connected, and $Q$ is negative definite \cite{BHPV}.  Since the resolution is
	minimal, $E$ contains no $(-1)$-curves.  Thus, writing $E_i^2=-e_i$, one has
	$e_i\geq2$ for every $i$.
	
	Set $H=-Q$.  Then $H$ is positive definite, with positive diagonal entries and
	nonpositive off-diagonal entries.  The connectedness of $E$ means that $H$ is
	irreducible.  Hence $H$ is an irreducible symmetric nonsingular $M$-matrix, and
	its inverse is strictly positive \cite{BermanPlemmons}:
	\begin{align}\label{eq:minimal-H-inverse-positive}
		(H^{-1})_{ij}>0
		\qquad
		\text{for all } i,j .
	\end{align}
	
	We next compute the relative canonical divisor.  Since
	$Y=\mathbb C^2/\Gamma$ is a quotient surface singularity, it is
	$\mathbb Q$-Gorenstein.  Indeed, a suitable positive tensor power of the standard
	holomorphic two-form on $\mathbb C^2$ is $\Gamma$-invariant and descends to the
	smooth part of $Y$.  Thus $K_Y$ is $\mathbb Q$-Cartier, and $\pi^*K_Y$ is defined
	as a $\mathbb Q$-divisor.  The difference $K_X-\pi^*K_Y$ is supported on the
	exceptional divisor, so we may write
	\begin{align}\label{eq:canonical-difference}
		K_X-\pi^*K_Y=\sum_{i=1}^k a_iE_i.
	\end{align}
	For each exceptional curve $E_j$, one has $\pi^*K_Y\cdot E_j=0$, because $E_j$ is
	contracted by $\pi$.  On the other hand, by adjunction and by $E_j^2=-e_j$,
	\begin{align*}
		K_X\cdot E_j
		=
		-2-E_j^2
		=
		e_j-2 .
	\end{align*}
	It follows that
	\begin{align*}
		\sum_i a_i(E_i\cdot E_j)=e_j-2
		\qquad
		\text{for every } j.
	\end{align*}
	Define $\mu=(\mu_1,\dots,\mu_k)^T$ by $H\mu=(e_1-2,\dots,e_k-2)^T$, since $H=-Q$, equation \eqref{eq:canonical-difference} becomes
	\begin{align}\label{eq:minimal-canonical-mu}
		K_X-\pi^*K_Y
		=
		-\sum_{i=1}^k\mu_iE_i .
	\end{align}
	
	Define the areas of the exceptional curves by
	\begin{align*}
		A_i
		:=
		\int_{E_i}\omega,
		\qquad
		i=1,\dots,k .
	\end{align*}
	Since $[\omega]$ is a K\"ahler class and each $E_i$ is an effective curve, one has
	$A_i>0$.
	
	We now compute the Chern--K\"ahler pairing appearing in Corollary~\ref{cor:Kahler-ALE-Weyl-HL}.
	
	\begin{lemma}\label{lem:Chern-Kahler-exceptional-pairing}
		For the scalar-flat K\"ahler ALE metric $g$ on $X$, one has
		\begin{align*}
			\left\langle
			\clubsuit(c_1(X,J)),[\omega]
			\right\rangle
			=
			\sum_{i=1}^k\mu_iA_i.
		\end{align*}
	\end{lemma}
	
	\begin{proof}
		The classes of the exceptional curves form a basis of $H_2(X;\R)$.  We use the
		intersection-matrix form of the Hein--LeBrun mass formula
		\cite[Theorem~5.3]{HL}.  In their notation, if
		$\Sigma_1,\dots,\Sigma_b$ is a basis of $H_2(M;\R)$, if
		$Q^\Sigma=(\Sigma_\alpha\cdot\Sigma_\beta)$, and if
		\begin{align*}
			(\lambda_1,\dots,\lambda_b)^T
			=
			(Q^\Sigma)^{-1}
			\left(
			\int_{\Sigma_1}c_1,\dots,
			\int_{\Sigma_b}c_1
			\right)^T,
		\end{align*}
		then the Chern--K\"ahler pairing in \eqref{eq:HL-mass-formula-dim2} is
		\begin{align*}
			\left\langle
			\clubsuit(c_1(X,J)),[\omega]
			\right\rangle
			=
			\sum_{\alpha=1}^b\lambda_\alpha\int_{\Sigma_\alpha}\omega.
		\end{align*}
		
		In the present case, take $\Sigma_i=E_i$.  Then $Q^\Sigma=Q$.  By adjunction,
		\begin{align*}
			\int_{E_j}c_1(X,J)
			=
			-K_X\cdot E_j
			=
			2-e_j.
		\end{align*}
		Therefore the Hein--LeBrun coefficient vector is
		\begin{align*}
			Q^{-1}
			(2-e_1,\dots,2-e_k)^T
			&=
			H^{-1}
			(e_1-2,\dots,e_k-2)^T
			=
			\mu.
		\end{align*}
		Thus
		\begin{align*}
			\left\langle
			\clubsuit(c_1(X,J)),[\omega]
			\right\rangle
			=
			\sum_{i=1}^k\mu_i\int_{E_i}\omega
			=
			\sum_{i=1}^k\mu_iA_i.
		\end{align*}
	\end{proof}
	
	Combining Lemma~\ref{lem:Chern-Kahler-exceptional-pairing} with
	Corollary~\ref{cor:Kahler-ALE-Weyl-HL} gives the formula which will be used in
	the crepancy criterion.
	
	\begin{proposition}\label{prop:minimal-resolution-Weyl-area}
		In every preferred coordinate system $x$ obtained from the holomorphic end
		induced by the resolution map $\pi:X\to\mathbb C^2/\Gamma$, one has
		\begin{align}\label{eq:minimal-resolution-Weyl-area}
			\Weyl_\infty^{(x)}(g)
			=
			-\frac{|\Gamma|}{\pi}
			\left(
			\sum_{i=1}^k\mu_iA_i
			\right)
			\mathcal W_{J_\infty^{(x)}}.
		\end{align}
	\end{proposition}
	
	We now use minimality.  Since $e_i\geq2$, the vector
	$(e_1-2,\dots,e_k-2)^T$ is nonnegative.  If this vector is zero, then $\mu$ vanishes.  If it is nonzero, then the
	strict positivity of $H^{-1}$ in \eqref{eq:minimal-H-inverse-positive} gives
	$\mu_i>0$ for every $i$.  Since every $A_i$ is positive, the sum
	$\sum_i\mu_iA_i$ can vanish only when $\mu=0$.
	
	This proves the third main theorem mentioned in Subsection~\ref{sec:main123}.
	
	\begin{proof}[Proof of Theorem~\ref{thm:main3}]
		By Proposition~\ref{prop:minimal-resolution-Weyl-area} and the fact that
		$\mathcal W_{J_\infty^{(x)}}\neq0$, the condition
		$\Weyl_\infty^{(x)}(g)=0$ is equivalent to
		\begin{align*}
			\sum_{i=1}^k\mu_iA_i=0 .
		\end{align*}
		By the positivity argument above, this holds if and only if $\mu=0$.  Finally,
		by \eqref{eq:minimal-canonical-mu}, the condition $\mu=0$ is equivalent to
		\begin{align*}
			K_X-\pi^*K_Y=0,
		\end{align*}
		namely $K_X=\pi^*K_Y$.  Since the formula
		\eqref{eq:minimal-resolution-Weyl-area} holds in every preferred coordinate
		system obtained from the holomorphic end, the vanishing condition is independent
		of the chosen preferred coordinates.
	\end{proof}
	
	We close the section with several representative examples illustrating
	Theorem~\ref{thm:main3}.  We use the negative Hirzebruch--Jung convention
	$\frac{q}{p}=[e_1,\dots,e_k]$, so that the corresponding cyclic resolution
	chain has self-intersections $-e_1,\dots,-e_k$.  The symbol
	$\frac{1}{q}(1,p)$ denotes the cyclic quotient generated by
	$(z_1,z_2)\mapsto(\zeta z_1,\zeta^p z_2)$, where $\zeta^q=1$ and
	$(p,q)=1$.  The symbol $D_4$ denotes the rational double point of type $D_4$,
	equivalently the quotient by the binary dihedral subgroup of order $8$ in
	$\mathrm{SU}(2)$.
	
	As a non-cyclic non-crepant example, we take one of the standard finite
	subgroups of $\mathrm U(2)$ appearing in the classification used in
	\cite{LV19}, namely
	$\Gamma=\phi(L(1,6)\times D_8^*)$.  Here
	$\phi:S^3\times S^3\to\mathrm{SO}(4)$ is the standard two-to-one homomorphism
	induced by the left and right quaternionic actions on
	$\mathbb H\simeq\mathbb R^4$.  The group $L(1,6)$ has order $6$, and
	$D_8^*$ has order $8$; hence $|\Gamma|=24$.  In the corresponding
	Brieskorn--Hirzebruch--Jung resolution graph, the central curve has
	self-intersection $-3$, and each of the three arms consists of a single
	$(-2)$-curve.
	
	For convenience, we write
	$\mathcal W_x=\mathcal W_{J_\infty^{(x)}}$.  The vector $\mu$ below is the
	one defined by $H\mu=(e_i-2)$, equivalently by
	\eqref{eq:minimal-canonical-mu}.  In the last non-cyclic example, we label the
	central curve by $E_0$ and the three end curves by $E_1,E_2,E_3$.  The table
	records the relevant resolution data and the corresponding behavior of the
	leading Weyl term.
	
	\begin{table}[htbp]
		\centering
		\small
		\renewcommand{\arraystretch}{1.35}
		\setlength{\tabcolsep}{5pt}
		\begin{tabularx}{\textwidth}{@{}p{0.20\textwidth}p{0.24\textwidth}p{0.34\textwidth}X@{}}
			\toprule
			Type
			&
			Example
			&
			Resolution data
			&
			Conclusion
			\\
			\midrule
			
			cyclic crepant
			&
			$\frac{1}{q}(1,q-1)$, $q\geq2$
			&
			$\frac{q}{q-1}=[2,\dots,2]$, hence $\mu=0$
			&
			$\Weyl_\infty^{(x)}=0$
			\\
			\addlinespace[0.35em]
			
			cyclic non-crepant
			&
			$\frac{1}{5}(1,2)$
			&
			$\frac{5}{2}=[3,2]$, and
			$\mu=(\frac{2}{5},\frac{1}{5})^T$
			&
			$\Weyl_\infty^{(x)}\neq0$
			\\
			\addlinespace[0.35em]
			
			non-cyclic crepant
			&
			$D_4$
			&
			one central $(-2)$-curve and three end $(-2)$-curves; hence $\mu=0$
			&
			$\Weyl_\infty^{(x)}=0$
			\\
			\addlinespace[0.35em]
			
			non-cyclic non-crepant
			&
			$\phi(L(1,6)\times D_8^*)$
			&
			one central $(-3)$-curve and three $(-2)$ end curves;
			central entry first:
			$\mu=(\frac{2}{3},\frac{1}{3},\frac{1}{3},\frac{1}{3})$
			&
			$\Weyl_\infty^{(x)}\neq0$
			\\
			\bottomrule
		\end{tabularx}
		\vspace{0.4em}
		\caption{Representative examples for Theorem~\ref{thm:main3}.}
		\label{tab:crepancy-Weyl-examples}
	\end{table}
	
	The two crepant examples have vanishing leading Weyl term.  Indeed, for
	$\frac{1}{q}(1,q-1)$ the Hirzebruch--Jung string consists entirely of
	$(-2)$-curves, and the same is true for the $D_4$ rational double point.  Hence
	the vector $(e_i-2)$ vanishes, so $\mu=0$.  Proposition
	\ref{prop:minimal-resolution-Weyl-area} therefore gives
	\begin{align*}
		\Weyl_\infty^{(x)}(g)=0 .
	\end{align*}
	
	The two non-crepant examples show the opposite behavior.  For the cyclic
	quotient $\frac{1}{5}(1,2)$, the listed value
	$\mu=(\frac{2}{5},\frac{1}{5})^T$ and $|\Gamma|=5$ give
	\begin{align*}
		\Weyl_\infty^{(x)}(g)
		=
		-\frac{1}{\pi}(2A_1+A_2)\mathcal W_x .
	\end{align*}
	Since $A_1,A_2>0$ and $\mathcal W_x\neq0$, this leading Weyl term is nonzero.
	
	For the non-cyclic example $\Gamma=\phi(L(1,6)\times D_8^*)$, with the central
	curve labeled by $E_0$, the listed value
	$\mu=(\frac{2}{3},\frac{1}{3},\frac{1}{3},\frac{1}{3})$ and $|\Gamma|=24$ give
	\begin{align*}
		\Weyl_\infty^{(x)}(g)
		=
		-\frac{8}{\pi}
		\left(
		2A_0+A_1+A_2+A_3
		\right)
		\mathcal W_x .
	\end{align*}
	Again this is nonzero because all K\"ahler areas $A_i$ are positive and
	$\mathcal W_x\neq0$.
	
	Thus the examples display both sides of Theorem~\ref{thm:main3}: the ADE,
	crepant cases have no leading Weyl contribution at infinity, whereas the
	non-crepant minimal quotient resolutions carry a non-removable leading Weyl
	term.

\end{document}